\documentclass [intlimits, 11pt]{article}

\usepackage[mathscr]{eucal}
\usepackage{amsmath}
\usepackage{amsthm}
\usepackage{amssymb}
\usepackage{amsfonts} 
\usepackage{mathrsfs}
\usepackage{amscd}
\usepackage{enumerate} 
\usepackage{color,wrapfig} 
\usepackage[pdftex]{graphicx}
\usepackage{layout}
\usepackage{float} 
\usepackage{mathtools} 
\usepackage{textcomp} 

\newtheorem{Theorem}[equation]{Theorem}
\newtheorem*{IntroTheorem}{Theorem}

\newtheorem{Lemma}[equation]{Lemma}
\newtheorem{Definition}[equation]{Definition}
\newtheorem{Remark}[equation]{Remark}

\newtheorem{Definition and Proposition}[equation]{Definition and Proposition}
\newtheorem{Definition and Theorem}[equation]{Definition and Theorem}
\newtheorem{Definition and Remark}[equation]{Definition and Remark}
\newtheorem{Definition and Lemma}[equation]{Definition and Lemma}
\newtheorem{Definition and Example}[equation]{Definition and Example}

\newtheorem{`Theorem'}[equation]{`Theorem'}

\newcommand{\purge}[1]{} 

\newcommand{\vsp}{\vspace{2mm}}

\def\epsilon{\varepsilon}
\def\phi{\varphi}

\newcommand{\al}{\alpha}
\newcommand{\be}{\beta}
\newcommand{\ga}{\gamma}

\newcommand{\ka}{\kappa}
\newcommand{\lam}{\lambda}

\newcommand{\si}{\sigma}

\newcommand{\De}{\Delta}

\newcommand{\gadot}{{\dot{\ga}}}

\newcommand{\phidot}{{\dot{\phi}}}

\newcommand{\xiti}{{\ti{\xi}}}

\newcommand{\xti}{{\ti{x}}}
\newcommand{\yti}{{\ti{y}}}
\newcommand{\zti}{{\ti{z}}}

\newcommand{\ubar}{\bar{u}}
\newcommand{\vbar}{\bar{v}}
\newcommand{\wbar}{\bar{w}}

\def\N{{\mathbb N}}

\def\R{{\mathbb R}}

\def\Z{{\mathbb Z}}

\newcommand{\mcA}{\mathcal A}
\newcommand{\mcB}{\mathcal B}
\newcommand{\mcC}{\mathcal C}

\newcommand{\mcF}{\mathcal F}
\newcommand{\mcG}{\mathcal G}

\newcommand{\mcM}{\mathcal M}

\newcommand{\mcV}{\mathcal V}
\newcommand{\mcW}{\mathcal W}
\newcommand{\mcX}{\mathcal X}
\newcommand{\mcY}{\mathcal Y}
\newcommand{\mcZ}{\mathcal Z}

\newcommand{\mcMhat}{\widehat{\mathcal M}}

\newcommand{\mfx}{\mathfrak x}

\newcommand{\mfxti}{\ti{{\mathfrak x}}}

\newcommand{\mfmhat}{{\hat{\mathfrak m}}}

\newcommand{\ti}{\tilde}
\newcommand{\x}{\times}
\newcommand{\del}{\partial}

\newcommand{\beq}{\begin{equation}}
\newcommand{\eeq}{\end{equation}}
\newcommand{\beqs}{\begin{equation*}}
\newcommand{\eeqs}{\end{equation*}}

\DeclareMathOperator{\Crit}{Crit}

\DeclareMathOperator{\grad}{grad}

\DeclareMathOperator{\Hom}{Hom}

\DeclareMathOperator{\Ind}{Ind}

\def\slashii#1{\setbox0=\hbox{$#1$}             
\dimen0=\wd0                                 
\setbox1=\hbox{\sl/} \dimen1=\wd1            
\ifdim\dimen0>\dimen1                        
\rlap{\hbox to \dimen0{\hfil\sl/\hfil}}   
#1                                        
\else                                        
\rlap{\hbox to \dimen1{\hfil$#1$\hfil}}   
\hbox{\sl/}                               
\fi}                                         %
\def\slashiii#1{\setbox0=\hbox{$#1$}#1\hskip-\wd0\hbox to\wd0{\hss\sl/\/\hss}}
\newcommand{\refmodulispacecompact}{Theorem \ref{modulispacecompact}}

\newcommand{\refalmoststrict}{Definition \ref{almoststrict}}

\newcommand{\refsubsectionmorsencat}{Subsection \ref{subsectionmorsencat}}

\newcommand{\refVnglobular}{Lemma \ref{Vnglobular}}

\newcommand{\refVW}{Remark \ref{VW}}

\newcommand{\refhomncat}{Theorem \ref{homncat}}

\newcommand{\refWncat}{Theorem \ref{Wncat}}

\newcommand{\refwfunctor}{Theorem \ref{wfunctor}}

\newcommand{\refhomfunctor}{Theorem \ref{homfunctor}}

\begin{document}

\title{On the image of the almost strict Morse $n$-category under almost strict $n$-functors}

\author{Sonja Hohloch \\ \'Ecole Polytechnique F\'ed\'erale de Lausanne (EPFL)}

\date{April 4, 2014}

\maketitle

\begin{abstract}
\noindent
In an earlier work, we constructed the almost strict Morse $n$-category $\mathcal X$ which extends Cohen $\&$ Jones $\&$ Segal's flow category. In this article, we define two other almost strict $n$-categories $\mathcal V$ and $\mathcal W$ where $\mathcal V$ is based on homomorphisms between real vector spaces and $\mathcal W$ consists of tuples of positive integers. The Morse index and the dimension of the Morse moduli spaces give rise to almost strict  $n$-category functors $\mathcal F : \mathcal X \to \mathcal V$ and $\mathcal G : \mathcal X \to \mathcal W$.
\end{abstract}


\section{Introduction}

The aim of the present paper is to gain a better understanding of the almost strict Morse $n$-category $\mcX$ introduced in Hohloch \cite{hohloch} whose construction is quite involved. Thus we now come up with another two almost strict $n$-categories $\mcV$ and $\mcW$ which retain some of the properties of $\mcX$, but are much more accessible. This imitates the idea of representation theory of groups where one studies homomorphisms (`representations') from a given group into a `nicer' group.

\vsp

The Morse $n$-category $\mcX$ extends the flow category introduced by Cohen $\&$ Jones $\&$ Segal \cite{cohen-jones-segal} whose objects are the critical points of a Morse function and whose morphisms are the Morse moduli spaces between critical points.

Roughly, the construction of $\mcX$ goes as follows. Let $M$ be a smooth compact $m$-dimensional manifold and $f_0 : M \to \R$ a Morse function, i.e. the Hessian $D^2f_0$ is nondegenerate on the set of critical points $\Crit(f_0)=\{x_0 \in M \mid Df_0(x)=0\}$. The Morse index $\Ind(x_0)$ of a critical point $x_0$ is given by the number of negative eigenvalues of $D^2f_0(x)$. We choose a `good' metric $g_0$ and consider the Morse moduli space $\mcM(x_0, y_0, f_0):=\mcM(x_0, y_0, f_0, g_0)$ consisting of negative gradient flow lines $\gadot(t) =- \grad_g f_0(\ga(t))$ from $x_0 \in \Crit(f_0)$ to $y_0 \in \Crit(f_0)$ which is a smooth manifold of dimension $\Ind(x_0) - \Ind(y_0)$. Dividing by the $\R$-action induced by the flow and suitably compactifying, we obtain the compact, unparametrized Morse moduli space $\mcMhat(x_0, y_0, f_0):= \overline{\mcM(x_0, y_0, f_0) \slash \R}$ which is a $(\Ind(x_0)-\Ind(y_0)-1)$-dimensional manifold with corners.

What we just did on $M$, we repeat on $\mcMhat(x_0, y_0, f_0)$ paying attention to the boundary strata: We pick a `good' Morse function $f_{1 \left[\begin{smallmatrix} x_0 \\ y_0 \end{smallmatrix} \right]}$ on $\mcMhat(x_0, y_0, f_0)$ whose gradient vector field is tangential to the boundary strata. `Good' means in this context that the Morse function needs to be compatible with the boundary structure of $\mcMhat(x_0, y_0)$ which consists of cartesian products of unparametrized Morse moduli spaces of certain critical points and that the negative gradient flow flows from higher dimensional strata to lower dimensional strata, but never back. Moreover, pick a suitable metric $g_{1 \left[\begin{smallmatrix} x_0 \\ y_0 \end{smallmatrix} \right]}$. Then, for $x_1$, $y_1 \in \Crit( f_{1 \left[\begin{smallmatrix} x_0 \\ y_0 \end{smallmatrix} \right]})$, the space $\mcMhat(x_1, y_1, f_{1 \left[\begin{smallmatrix} x_0 \\ y_0 \end{smallmatrix} \right]})$ is again a manifold with corners on which we can choose a `good' Morse function $f_{2 \left[\begin{smallmatrix} x_1 x_0 \\ y_1  y_0 \end{smallmatrix} \right]}$ and iterate again.
Since the dimension drops at least by one when passing from $M$ to $\mcMhat(x_0, y_0, f_0)$ and also from $\mcMhat(x_0, y_0, f_0)$ to $\mcMhat(x_1, y_1, f_{1 \left[\begin{smallmatrix} x_0 \\ y_0 \end{smallmatrix} \right]})$ the iteration procedure terminates after a finite number of steps.

Obviously, this construction depends on the choice of a family of Morse functions $F:=\{f_0, f_{1 \left[\begin{smallmatrix} x_0 \\ y_0 \end{smallmatrix} \right]}, \dots\}$ and metrics $G:= \{g_0, g_{1 \left[\begin{smallmatrix} x_0 \\ y_0 \end{smallmatrix} \right]}, \dots\}$.

\vsp

Since the definition of an almost strict $n$-category is quite lengthy we do not recall it here in the introduction, but refer the reader to \refalmoststrict.

\begin{IntroTheorem} [Hohloch \cite{hohloch}]
The above described iteration of Morse moduli spaces admits the structure of an almost strict $n$-category. The resulting $n$-category is denoted by $\mcX:=\mcX(F,G)$.
\end{IntroTheorem}

In a future project, we will investigate the dependence of $\mcX$ on the choices $F$ and $G$. But in the present paper, we fix a choice $\mcX :=\mcX(F,G)$ and look for other less complicated almost strict $n$-categories which pertain some of the information of $\mcX$. 

More precisely, we will define an almost strict $n$-category $\mcV$ roughly consisting of tuples of spaces of linear maps and a functor of almost strict $n$-categories $\mcF : \mcX \to \mcV$ which maps a moduli space $\mcMhat(x_l, y_l, f_{l \left[\begin{smallmatrix} x_{l-1}, \dots, x_0 \\ y_{l-1}, \dots, y_0 \end{smallmatrix} \right]})$ to $\Hom(\R^{\Ind(x_l)}, \R^{\Ind(y_l)}) \times \cdots \times \Hom(\R^{\Ind(x_1)}, \R^{\Ind(y_1)}) $.

\vsp

There is another almost strict $n$-category $\mcW$ which roughly consists of tuples of natural numbers (including zero) and a functor of almost strict $n$-categories $\mcG: \mcX \to \mcW$ which maps a moduli space $\mcMhat(x_l, y_l, f_{l \left[\begin{smallmatrix} x_{l-1}, \dots, x_0 \\ y_{l-1}, \dots, y_0 \end{smallmatrix} \right]})$ to the tuple $\left[\begin{smallmatrix} \Ind(x_{l}), \dots, \Ind( x_0) \\ \Ind(y_{l}), \dots, \Ind(y_0) \end{smallmatrix} \right]$.

\begin{IntroTheorem}
$\mcV$ and $\mcW$ are almost strict $n$-categories.
\end{IntroTheorem}

This statement is proven in \refhomncat\ and \refWncat. For the definition of an almost strict $n$-category functor, we refer the reader to \refalmoststrict.

\begin{IntroTheorem}
There are almost strict $n$-category functors $\mcF : \mcX \to \mcV$ and $\mcG : \mcX \to \mcW$ which are based on the dimension of the Morse moduli spaces and the Morse index.
\end{IntroTheorem}

This is proven in \refhomfunctor\ and \refwfunctor.

\subsection*{Organization of the paper}

In Section 2, we recall notions and definitions associated to almost strict $n$-categories. In Section 3, we briefly sketch the idea of Morse moduli spaces on smooth manifolds and on manifolds with corners before we recall the almost strict Morse $n$-category $\mcX$ from Hohloch \cite{hohloch}. In section 4, we construct the almost strict $n$-categories $\mcV$ and $\mcW$ and define the $n$-category functors $\mcF : \mcX \to \mcV$ and $\mcG : \mcX \to \mcW$. In Section 5, we compute (an example of) $\mcX$ and its image under the functors on the 2-torus.


\section{Almost strict $n$-categories}

Strict $n$-categories were originally introduced by Ehresmann. We will use the formulation and conventions of Leinster's book \cite{leinster}.

\begin{Definition}
\label{nglobular}
Given $n \in \N$, we define an \textbf{$n$-globular set} $Y$ to be a collection of sets $\{Y(l) \mid {0 \leq l \leq n} \}$ together with \textbf{source} and \textbf{target functions} $s$, $t: Y(l) \to Y(l-1)$ for $1 \leq l \leq n$ satisfying $s \circ s = s \circ t$ and $t \circ s = t \circ t$. Elements $A_l \in Y(l)$ are called \textbf{$l$-cells}.
\end{Definition}

To visualize $n$-globular sets, one can think of $l$-cells as $l$-dimensional disks like in Figure \ref{cells}: (a) shows a $0$-cell $A_0 \in Y(0)$, (b) displays a $1$-cell $A_1 \in Y(1)$ with $s(A_1)=A_0 \in Y(0)$ and $t(A_1)=B_0 \in Y(0)$, (c) sketches a $2$-cell $A_2 \in Y(2)$ with $s(A_2)=A_1$, $t(A_2)=B_1 \in Y(1)$ and therefore $s(A_1)=s(B_1)=A_0$ and $t(A_1)=t(B_1)=B_0$.

\begin{figure}[h]
\begin{center}

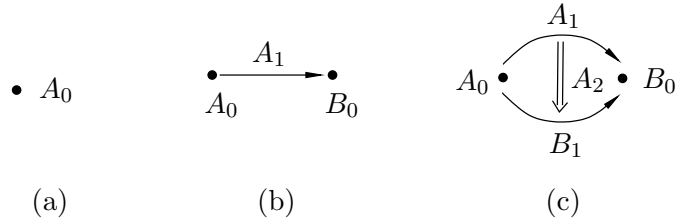

\caption{(a) $0$-cell, (b) $1$-cell, (c) $2$-cell.}
\label{cells}

\end{center}
\end{figure}

If we want to compose two $l$-cells $\mfx$ and $\mfxti$ along a $p$-cell, we need certain matching conditions which are described by the set
\beqs
Y(l) \x_p Y(l) := \{(\mfxti, \mfx) \in Y(l) \x Y(l) \mid s^{l-p}(\mfxti)=t^{l-p}(\mfx)\}
\eeqs
where $0 \leq p < l \leq n$.

\begin{Definition}
\label{ncategory}
Let $n \in \N$. A \textbf{strict $n$-category $\mcY$} is an $n$-globular set $Y$ equipped with
\begin{itemize}
 \item 
a function $\circ_p: Y(l) \x_p Y(l) \to Y(l)$ for all $0 \leq p < l\leq n$. We set $\circ_p(C_l, A_l) =:C_l \circ_p A_l$ and call it \textbf{composite} of $A_l$ and $C_l$.
\item
a function $\mathbf{1}: Y(l) \to Y(l+1) $ for all $0 \leq l < n$. We set $\mathbf{1}_{A_l}:=\mathbf{1}(A_l)$ and call it the \textbf{identity} on $A_l$.
\end{itemize}
These have to satisfy the following axioms:
\begin{enumerate}[(a)]
 \item 
\textbf{(Sources and targets of composites)} For $0 \leq p < l \leq n$ and $(C_l, A_l) \in Y(l) \x_p Y(l) $ we require
\begin{align*}
& \mbox{for } p=l-1: \ && s(C_l \circ_p A_l)=s(A_l) && \mbox{and} && t(C_l \circ_p A_l) = t(C_l), \\ 
& \mbox{for } p \leq l-2: \ && s(C_l \circ_p A_l)= s(C_l) \circ_p s(A_l) && \mbox{and} && t(C_l \circ_p A_l)=t(C_l) \circ_p t(A_l).
\end{align*}
\item
\textbf{(Sources and targets of identities)} For $0 \leq l <  n$ and $A_l \in Y(l)$ we require 
\beqs
s(\mathbf{1}_{A_l})=A_l=t(\mathbf{1}_{A_l}).
\eeqs
\item
\textbf{(Associativity)}  
For $0 \leq p < l \leq n$ and $A_l$, $C_l$, $E_l\in Y(l)$ with $(E_l, C_l)$, $(C_l, A_l) \in Y(l) \x_p Y(l)$ we require
\beqs
(E_l\circ_p C_l) \circ_p A_l = E_l \circ_p( C_l \circ_p A_l).
\eeqs 
\item \textbf{(Identities)}
For $0 \leq p < l \leq n$ and $A_l \in Y(l)$ we require 
\beqs
\mathbf{1}^{l-p}(t^{l-p}(A_l)) \circ_p A_l = A_l = A_l \circ_p \mathbf{1}^{l-p}(s^{l-p}(A_l)).
\eeqs
\item \textbf{(Binary interchange)}
For $0 \leq q < p < l \leq n$ and $A_l$, $C_l$, $E_l$, $H_l \in Y(l)$ with
\beqs
(H_l, E_l), (C_l, A_l) \in Y(l) \x_p Y(l) \mbox{ and } (H_l, C_l), (E_l,A_l) \in Y(l)\x_q Y(l)
\eeqs
we require
\beqs
(H_l \circ_p E_l) \circ_q (C_l \circ_p A_l) = (H_l \circ_q C_l) \circ_p (E_l \circ_q A_l).
\eeqs
\item \textbf{(Nullary interchange)}
For $0 \leq p < l< n$ and $(C_l, A_l) \in Y(l) \x_p Y(l)$ we require $\mathbf{1}_{C_l} \circ_p \mathbf{1}_{A_l} = \mathbf{1}_{C_l \circ_p A_l}$.
\end{enumerate}
If $\mcY$ and $\mcZ$ are strict $n$-categories we define a \textbf{strict $n$-functor $f$} as a map $f : Y \to Z$ of the underlying $n$-globular sets commuting with composition and identities. This defines a category \textbf{Str-$n$-Cat} of strict $n$-categories.
\end{Definition}

If we slightly relax the requirements, we get

\begin{Definition}
\label{almoststrict}
An \textbf{almost strict $n$-category} satisfies the requirements of a strict $n$-category up to canonical isomorphism.
Let $\mcA$ and $\mcB$ be two almost strict $n$-categories with $n$-globular sets $A$ and $B$. An \textbf{\textit{almost strict $n$-category functor}}, briefly an \textbf{\textit{$n$-functor}}, $\mcF : \mcA \to \mcB$ is a map $\mcF : A \to B$ of the underlying $n$-globular sets commuting with composition and identities. This defines the category $\mcC$ of almost strict $n$-categories.
\end{Definition}


\section{The almost strict Morse $n$-category}


\subsection{Morse moduli spaces on smooth manifolds without boundary}

In the following, we are interested in the dynamical approach to Morse theory via the negative gradient flow of a Morse function as described for instance by Schwarz \cite{schwarz}.

\vsp 

Let $M$ be a closed $m$-dimensional manifold. A smooth function $f : M \to \R$ is a {\em Morse function} if its Hessian $D^2f$ is nondegenerate at the critical points $\Crit(f):=\{ x \in M \mid Df(x)=0\}$. The {\em Morse index} $\Ind(x)$ of a critical point $x$ is the number of negative eigenvalues of $D^2f(x)$. Given a Riemannian metric $g$ on M, we denote by $\grad_{g}f$ the gradient of $f$ w.r.t. the metric $g$.
The autonomous ODE of the {\em negative gradient flow} $\phi_t$ of the pair $(f,g)$ is given by
\beqs
\phidot_t = - \grad_g f(\phi_t).
\eeqs
The {\em stable manifold} of a critical point $x \in \Crit(f)$ is
\beqs
W^s(f,x):= W^s(f,g,x):= \{ p \in M \mid  \lim_{t \to + \infty} \phi_t(p)=x\}
\eeqs
and the {\em unstable manifold} is
\beqs
W^u(f,x):= W^u(f,g,x):= \{ p \in M \mid \lim_{t \to - \infty} \phi_t(p) =x\}.
\eeqs
A pair $(f,g)$ is called {\em Morse-Smale} if $W^s(f,g,x)$ and $W^u(f,g,y)$ intersect transversely for all $x$, $y \in \Crit(f)$. 
The {\em Morse moduli space} between two critical points $x$ and $y$ is the space of smooth curves
\beqs
\mcM(x,y):= \mcM(x,y, f,g):= 
\left\{
 \ga: \R \to M \left|
\begin{aligned}
 & \gadot(t) = - \grad_g f(\ga(t)), \\ 
 & \lim_{t \to - \infty} \ga(t)=x , \\ 
 & \lim_{t \to + \infty} \ga(t)=y
\end{aligned}
\right.
\right\}.
\eeqs
This are the negative gradient flow lines running from $x$ to $y$. It can also be identified with $W^u(x,f) \cap W^s(f,y)$. 
If $(f,g)$ is Morse-Smale $\mcM(x,y)$ is a smooth manifold of dimension $\Ind(x)- \Ind(y)$. If $\Ind(y)>\Ind(x)$ then the space $\mcM(x,y)$ is empty.
For $\ga \in \mcM(x,y)$ and $\si \in \R$, the curve $\ga_\si$ with $\ga_\si(t):= \ga(t+\si)$ is also a gradient flow line. Thus there is an action $\R \x \mcM(x,y) \to \mcM(x,y)$, $(\ga, \si) \mapsto \ga_\si$. Dividing by the action, we obtain the {\em unparametrized} moduli space $\mcM(x,y)\slash \R$. 

\vsp

In order to describe Morse moduli spaces properly we need the notion of a manifold with corners. An overview over the various definitions of manifolds with corners and their differences may be found in Joyce \cite{joyce} whose conventions we will use. An \emph{$m$-dimensional manifold with corners} is an $m$-dimensional manifold which is locally modeled on $\R^m_+:=(\R_{\geq 0})^m$. Let $\psi=(\psi_1, \dots, \psi_m) : U \subseteq N \to \R^m_+$ be a chart of an $m$-dimensional manifold with corners $N$.
For $x \in U$, set
\beqs
depth(x):=\#\{i \mid \psi_i(x)=0,\ 1 \leq i \leq m \}.
\eeqs
A {\em face} of $N$ is the {\em closure} of a connected component of $\{x \in N \mid depth(x)=1\}$. 
If $k$ is the number of faces, we fix an {\em order} of the faces and denote them by $\del_1 N, \dots, \del_k N$. 
The connected components of $\{x \in N \mid depth(x)=l\}=:D_{\dim N -l}$ are called the {\em ($\dim N -l$)-strata} of $N$.

\begin{Definition}
Let $N$ be an $m$-dimensional manifold with corners with $k$ faces $\del_1 N, \dots, \del_k N$. We call $N$ a \textbf{$\langle k \rangle$-manifold} if 
\begin{enumerate}[(a)]
 \item 
Each $x \in N$ lies in $depth(x)$ faces.
\item
$\del_1 N \cup \dots \cup \del_k N = \del N$.
\item
For all $1 \leq i,j \leq k$ with $i \neq j$ the intersection $\del_i N \cap \del_j N $ is a face of both $\del_i N$ and $\del_j N$. 
\end{enumerate}
\end{Definition}

Here $\del_i N \subset N$ is again a manifold with corners, but $\del N$ is not. We stick to Joyce's \cite{joyce} definition where the integer $\langle k \rangle$ has a priori nothing to do with the dimension $m$ of the manifold $N$. Other authors like Laures \cite{laures} let $\del_i N$ be a union of faces which admits $k= \dim N$.
An example of a $\langle k \rangle$-manifold is $\R^k_+$ with faces $\del_i\R^k_+:= \{ x \in \R^k_+ \mid x_i=0\}$. $\langle 0 \rangle$-manifolds are manifolds without boundary and $\langle 1 \rangle$-manifolds are manifolds with one (smooth) boundary component.

\vsp

Now let $M$ be a smooth compact manifold with Morse-Smale pair $(f,g)$ and $x$, $y$, $z \in \Crit(f)$ with $\Ind(x) > \Ind(y) > \Ind(z)$. 
Figure \ref{breaking} (a) displays a sequence of trajectories $(\ga_n)_{n \in \N}$ from $x$ to $z$ which `break' in the limit into trajectories $\ga_{xy}$ from $x$ to $y$ and $\ga_{yz}$ from $y$ to $z$. This phenomenon is called `breaking' and plays an important role if one wants to compactify unparametrized Morse moduli spaces.
More precisely, one usually compactifies an unparametrized moduli space by adding `broken trajectories' as boundary points. We denote this compactification of $\mcM(x,z)\slash \R$ via adding broken trajectories by $\mcMhat(x,z):= \overline{\mcM(x,z)\slash \R}$. 
In order to obtain a nice structure on the compactification one needs to pose conditions on the metric.
If $f$ is a Morse function and if a metric $g$ is euclidean near the critical points of $f$ the we call $g$ an {\em $f$-euclidean metric}.

\begin{figure}[h]

\begin{center}

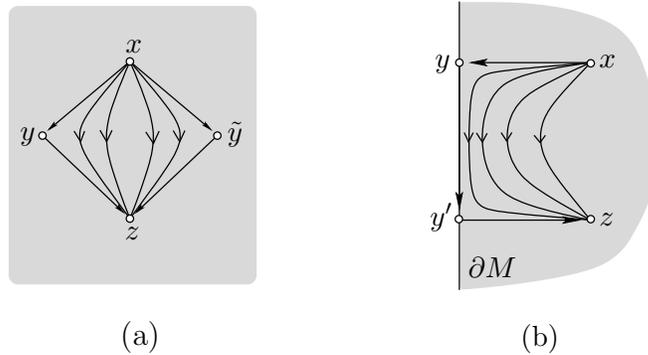
\caption{Breaking of trajectories: (a) in the interior, (b) on the boundary.}
\label{breaking}
\end{center}

\end{figure}

For $x$, $y \in \Crit(f)$ with $x \neq y$, we introduce the notation $x>y$ if $\mcM(x,y) \neq \emptyset$.

\begin{Theorem}[\cite{burghelea}, \cite{wehrheim}, \cite{qin1}, \cite{qin2}]
\label{modulispacecompact}
Let $M$ be compact and $(f,g)$ be Morse-Smale and assume $g$ to be $f$-euclidean. Let $x$, $z \in \Crit(f)$ with $x>z$. Then there exists $k \in \N_0$ such that $\mcMhat(x,z)$ is an $(\Ind(x)-\Ind(z)-1)$-dimensional $\langle k \rangle$-manifold with corners and its boundary is given by 
\beqs
\del \mcMhat(x,z) = \bigcup_{\stackrel{ (\Ind(x) - \Ind(z) -1) \geq l \geq 0}{x > y_1 > \dots > y_l >z}} \mcMhat(x,y_1) \x\mcMhat(y_1, y_2) \x \dots \x \mcMhat(y_{l-1}, y_l) \x \mcMhat(y_l, z)
\eeqs
where $y_1$, \dots, $y_l \in \Crit(f)$. There is a canonical smooth structure on $\mcMhat(x,z)$.
\end{Theorem}

The `inverse procedure' of breaking is `gluing' which takes a broken trajectory $(\ga_{xy}, \ga_{yz}) \in \mcMhat(x,y) \x \mcMhat(y,z)$ and `glues' it to a Morse trajectory from $x$ to $z$.
Gluing multiply broken trajectory $(\ga_1, \dots, \ga_{l+1}) \in \mcMhat(x,y_1) \x \dots \x \mcMhat(y_l, z)$ is well defined since gluing is associative  (cf.\ Qin \cite{qin3} and Wehrheim \cite{wehrheim}).


\subsection{Morse moduli spaces on $\langle k \rangle$-manifolds}

\label{modulicorners}

For manifolds with smooth boundary, there is a Morse theory approach via the gradient flow tangential to the boundary (cf.\ Akaho \cite{akaho}, Kronheimer $\&$ Mrowka \cite{kronheimer-mrowka}). Ludwig \cite{ludwig} defined Morse theory with tangential gradient vector field on stratified spaces.

\vsp

Let $M$ be a smooth compact manifold. 
Let $(f_0, g_0)$ be a Morse-Smale pair consisting of a Morse function $f_0$ with $f_0$-euclidean metric $g_0$. Let $x_0$, $z_0 \in \Crit(f_0)$ be distinct critical points and consider $\mcMhat(x_0, z_0, f_0)$. If this moduli space is not empty then, by \refmodulispacecompact, it is a manifold (possibly) with corners whose boundary is of the form
\beqs
\del \mcMhat(x_0,z_0,f_0) = \bigcup_{\stackrel{ (\Ind(x_0) - \Ind(z_0) -1) \geq l \geq 0}{x_0 > y^1_0 > \dots > y^l_0 >z_0}} \mcMhat(x_0,y^1_0, f_0) \x \dots \x \mcMhat(y^l_0, z_0, f_0)
\eeqs
where $y^1_0$, \dots, $y^l_0 \in \Crit(f_0)$. Using this formula recursively we can also write
\beqs
\del \mcMhat(x_0, z_0, f_0)= \bigcup_{y_0 \in \Crit(f_0)} \mcMhat(x_0, y_0, f_0) \x \mcMhat(y_0, z_0, f_0).
\eeqs
A moduli space may have several connected components. Choosing an ordering for the components of depth one, we endow $\mcMhat(x_0, z_0, f_0)$ with the structure of a $\langle k \rangle $-manifold for some $k \in \N_0$. Note that $\mcMhat(x_0, z_0, f_0)$ might share strata with other moduli spaces $\mcMhat(\xti_0, \zti_0, f_0)$ for $\xti_0$, $\zti_0 \in \Crit(f_0)$.

\begin{Theorem}[Hohloch \cite{hohloch}]
\label{morseOnManifoldsWithCorners}
Let $f$ be a Morse function on a compact $\langle k \rangle$-manifold whose negative gradient flow is tangential to the boundary strata and flows from higher to lower strata, but not from lower to higher ones. Assume the metric to be euclidean near the critical points. 
 Let $x$, $z \in \Crit(f)$ with $x>z$. Then there exists $k \in \N_0$ such that $\mcMhat(x,z)$ is an $(\Ind(x)-\Ind(z)-1)$-dimensional $\langle k \rangle$-manifold with corners and its boundary is given by 
\beqs
\del \mcMhat(x,z) = \bigcup_{\stackrel{ (\Ind(x) - \Ind(z) -1) \geq l \geq 0}{x > y_1 > \dots > y_l >z}} \mcMhat(x,y_1) \x\mcMhat(y_1, y_2) \x \dots \x \mcMhat(y_{l-1}, y_l) \x \mcMhat(y_l, z)
\eeqs
where $y_1$, \dots, $y_l \in \Crit(f)$. There is a canonical smooth structure on $\mcMhat(x,z)$.
\end{Theorem}


\subsection{The almost strict Morse $n$-category}

\label{subsectionmorsencat}

In this subsection, we assume that all Morse functions satisfy:
\begin{enumerate}[1)]
 \item 
Their gradient vector field is tangential to the boundary strata.
  \item
The Morse function is compatible with the cartesian product structure of the boundary of a Morse moduli space.
  \item
The negative gradient flow only flows from higher dimensional into lower dimensional strata, but never from lower to higher dimensional strata, i.e.\ a behaviour like in Figure \ref{breaking} (b) is prevented. 
\end{enumerate}
For the existence and construction of such Morse functions, we refer to Hohloch \cite{hohloch}.

\vsp

In the following, we summarize the construction of the almost strict Morse $n$-category from the earlier work Hohloch \cite{hohloch}.
Let $M$ be a compact $n$-dimensional $\langle k \rangle$-manifold $M$ with a Morse function $f_0$ and an $f_0$-euclidean metric $g_0$. We set 
\beqs
X(0):=\{x_0 \mid x_0 \in \Crit(f_0)\}.
\eeqs
Let $x_0$, $y_0 \in \Crit(f_0)$ and choose on the space $\mcMhat(x_0, y_0, f_0)$ a Morse function
$f_{1 \left[ \begin{smallmatrix} x_0 \\ y_0 \end{smallmatrix}\right]}$ with $f_{1 \left[ \begin{smallmatrix} x_0 \\ y_0 \end{smallmatrix}\right]}$-euclidean metric $g_{1 \left[ \begin{smallmatrix} x_0 \\ y_0 \end{smallmatrix}\right]}$. We define
\beqs
X(1):=\{(x_1, \mcMhat(x_0, y_0, f_0)) \mid x_0, y_0 \in \Crit(f_0), \ x_1 \in \Crit(f_{1 \left[ \begin{smallmatrix} x_0 \\ y_0 \end{smallmatrix}\right]})\}.
\eeqs
The index of the Morse function $f_{1 \left[ \begin{smallmatrix} x_0 \\ y_0 \end{smallmatrix}\right]}$ or metric $g_{1 \left[ \begin{smallmatrix} x_0 \\ y_0 \end{smallmatrix}\right]}$ starts with the number of the `iteration level' on which the function or metric lives and continues with the (history of) critical points which gave rise to the moduli space. The upper row contains the source points and the lower row the target points. To remember the `history' of a moduli space is essential.
Iterating leads to 
\beqs
X(l):=
\left\{
\left(x_l, \mcMhat(x_{l-1}, y_{l-1}, 
f_{l-1 
\left[
\begin{smallmatrix}
x_{l-2}, \dots, x_0 \\
y_{l-2}, \dots, y_0
\end{smallmatrix}
\right]
}
)\right)
\left|
\begin{aligned}
& 0 \leq j \leq l-1, \\
& x_j, y_j \in \Crit(
f_{j 
\left[
\begin{smallmatrix}
x_{j-1}, \dots, x_0 \\
y_{j-1}, \dots, y_0
\end{smallmatrix}
\right]
}
), \\
& x_l \in \Crit(f_
{l 
\left[
\begin{smallmatrix}
x_{l-1}, \dots, x_0 \\
y_{l-1}, \dots, y_0
\end{smallmatrix}
\right]
}
)
\end{aligned}
\right.
\right\}
\eeqs
for $2 \leq l \leq n$.
We define {\em source} and {\em target functions} 
\beqs
s: X(l) \to X(l-1) \qquad \mbox{and} \qquad t: X(l) \to X(l-1)
\eeqs
for $2 \leq l \leq n$ via
\begin{align*}
s\left(x_l, \mcMhat(x_{l-1}, y_{l-1}, 
f_{l-1 
\left[
\begin{smallmatrix}
x_{l-2}, \dots, x_0 \\
y_{l-2}, \dots, y_0
\end{smallmatrix}
\right]
}
)\right)
&:= 
\left(x_{l-1}, \mcMhat(x_{l-2}, y_{l-2}, 
f_{l-2 
\left[
\begin{smallmatrix}
x_{l-3}, \dots, x_0 \\
y_{l-3}, \dots, y_0
\end{smallmatrix}
\right]
}
)\right), \\
t\left(x_l, \mcMhat(x_{l-1}, y_{l-1}, 
f_{l-1 
\left[
\begin{smallmatrix}
x_{l-2}, \dots, x_0 \\
y_{l-2}, \dots, y_0
\end{smallmatrix}
\right]
}
)\right)
&:= 
\left(y_{l-1}, \mcMhat(x_{l-2}, y_{l-2}, 
f_{l-2 
\left[
\begin{smallmatrix}
x_{l-3}, \dots, x_0 \\
y_{l-3}, \dots, y_0
\end{smallmatrix}
\right]
}
)\right) \\
\end{align*}
and set for $s, t: X(1) \to X(0)$ 
\beqs
s\left (a_1, \mcMhat(x_0, y_0, f_0)\right ):= x_0 \quad \mbox{and} \quad t\left (a_1, \mcMhat(x_0, y_0, f_0)\right):= y_0.
\eeqs

\begin{Lemma}[Hohloch \cite{hohloch}]
$X:=\{X(l) \mid 0 \leq l \leq n\}$ is an $n$-globular set.
\end{Lemma}

The $l$-cells which can be composed along $p$-cells are described by
\beqs
X(l)\x_p X(l):=\{(C_l, A_l) \in X(l) \x X(l) \mid s^{l-p}( C_l)=t^{l-p}(A_l)  \}.
\eeqs
Tuples $(C_l, A_l) \in X(l) \x_p X(l)$ can be displayed via
\begin{align*}
A_l&= \left(a_l,\mcMhat(a_{l-1}, b_{l-1}, f_{l-1 
\left[
\begin{smallmatrix}
a_{l-2}, \dots, a_{p+1}, x_p, \al_{p-1}, \dots, \al_0 \\
b_{l-2}, \dots, b_{p+1}, y_p, \be_{p-1}, \dots, \be_0
\end{smallmatrix}
\right]
})\right), \\
C_l &= \left( c_l,\mcMhat(c_{l-1}, d_{l-1}, f_{l-1 
\left[
\begin{smallmatrix}
c_{l-2}, \dots, c_{p+1}, y_p, \al_{p-1}, \dots, \al_0 \\
d_{l-2}, \dots, d_{p+1}, z_p, \be_{p-1}, \dots, \be_0
\end{smallmatrix}
\right]
})\right) .
\end{align*}
Being in $ X(l) \x_p X(l)$ means the following: Both $l$-cells arise, up to level $(p-1)$, from the same critical points 
$\left[
\begin{smallmatrix}
\al_{p-1}, \dots, \al_0 \\
\be_{p-1}, \dots, \be_0
\end{smallmatrix}
\right]$.
At level $p$, we have the matching condition 
$\begin{smallmatrix}
\left[
\begin{smallmatrix}
 x_p \\
 y_p
\end{smallmatrix}
\right] \\
\left[
\begin{smallmatrix}
 y_p \\
 z_p
\end{smallmatrix}
\right]
\end{smallmatrix}.
$
There are no additional conditions on the critical points on the higher levels 
$
\begin{smallmatrix}
\left[
\begin{smallmatrix}
a_{l-2} , \dots, a_{p+1} \\
b_{l-2}, \dots, b_{p+1}
\end{smallmatrix}
\right] \\
\left[
\begin{smallmatrix}
c_{l-2}, \dots, c_{p+1} \\
d_{l-2}, \dots, d_{p+1}
\end{smallmatrix}
\right]
\end{smallmatrix}
$
apart from the ones required in the definition of $X(l)$. The tuple 
$
\left[
\begin{smallmatrix}
a_{l-2}, \dots, a_{p+1}, x_p, \al_{p-1}, \dots, \al_0 \\
b_{l-2}, \dots, b_{p+1}, y_p, \be_{p-1}, \dots, \be_0
\end{smallmatrix}
\right]
$
is the {\em history} of $A_l$ up to level $(l-2)$.
If $j=1$ in the two expressions above then there are no $a$'s and $b$'s resp. $c$'s and $d$'s in the index of the function.

\vsp 

The identity functions $\mathbf{1} : X(l) \to X(l+1)$ are defined as follows. 
Let $x_0 \in X(0)$ 
and identify $x_0$ with the moduli space $\mcMhat(x_0, x_0, f_0)$. Then identify $\mcMhat(x_0, x_0, f_0)$ with the only critical point 
$x_1\in \Crit(f_{1 
\left[ 
\begin{smallmatrix}
x_0 \\                                                                                                                                                                  
x_0                                                                                                                                                                 \end{smallmatrix}
 \right]})$ 
on $\mcMhat(x_0, x_0, f_0) $. Thus we have $x_1 \simeq \mcMhat(x_0, x_0, f_0) \simeq x_0$ and we set
\beqs
\mathbf{1}_{x_0}:=\mathbf{1}(x_0):= (x_0, \mcMhat(x_0, x_0, f_0)).
\eeqs
For $l>0$, we set for $A_l= \left(a_l, \mcMhat(a_{l-1}, b_{l-1}, f_{l-1
\left[
\begin{smallmatrix}
 a_{l-2}, \dots, a_0 \\
b_{l-2}, \dots, b_0
\end{smallmatrix}
\right]}
)\right) \in X(l)$
\begin{align*}
\mathbf{1}_{A_l}
& := \mathbf{1}\left (a_l, \mcMhat(a_{l-1}, b_{l-1}, f_{l-1
\left[
\begin{smallmatrix}
 a_{l-2}, \dots, a_0 \\
b_{l-2}, \dots, b_0
\end{smallmatrix}
\right]}
)\right) \\
& := \left(a_l, \mcMhat(a_l, a_l, f_{l
\left[
\begin{smallmatrix}
 a_{l-1}, \dots, a_0 \\
b_{l-1}, \dots, b_0
\end{smallmatrix}
\right]
})\right) \\
& := \left(a_{l+1}, \mcMhat(a_l, a_l, f_{l
\left[
\begin{smallmatrix}
 a_{l-1}, \dots, a_0 \\
b_{l-1}, \dots, b_0
\end{smallmatrix}
\right]
})\right)
\end{align*}
where we again identified $a_{l+1} \simeq a_l$. For $0 \leq l \leq n-1$, this gives us functions 
\beqs
\mathbf{1} : X(l) \to X(l+1). 
\eeqs

The composite $\circ_p$ for $l > p \geq 0$ is defined as follows.


\subsubsection*{Case $l \in \N$ and $p=0$:}

There are no $\al$'s and $\be$'s such that the `history index' starts with $x_0$, $y_0$, $z_0$. We set
\begin{align*}
 & \left( c_l, \mcMhat(c_{l-1}, d_{l-1}, f_{l-1 \left[ \begin{smallmatrix} c_{l-2}, \dots, c_1, y_0 \\ d_{l-2}, \dots, d_1, z_0 \end{smallmatrix} \right]}) \right)
\circ_0
\left(
a_l, \mcMhat(a_{l-1}, b_{l-1}, f_{l-1 \left[ \begin{smallmatrix} a_{l-2}, \dots, a_1, x_0 \\ b_{l-2}, \dots, b_1, y_0 \end{smallmatrix} \right]})
\right) \\
& :=
\left(
(a_l, c_l), \mcMhat\bigl((a_{l-1}, c_{l-1}), (b_{l-1}, d_{l-1}), f_{l-1 \left[ \begin{smallmatrix} (a_{l-2}, c_{l-2}), \dots, (a_1, c_1), x_0 \\ (b_{l-2}, d_{l-2}), \dots, (b_1, d_1), z_0 \end{smallmatrix} \right]}\bigr)
\right).
\end{align*}


\subsubsection*{Case $l \in \N$ and $l-2 \geq p \geq 1$:}

We set
\begin{align*}
& \left(
c_l, \mcMhat(c_{l-1}, d_{l-1}, f_{l-1 \left[ 
\begin{smallmatrix} 
c_{l-2}, \dots, c_{p+1}, y_p, \al_{p-1}, \dots, \al_0 \\
d_{l-2}, \dots, d_{p+1}, z_p, \be_{p-1}, \dots, \be_0
\end{smallmatrix} \right]})
\right) \\
& \quad \circ_p
\left(
a_l, \mcMhat(a_{l-1}, b_{l-1}, f_{l-1 \left[ 
\begin{smallmatrix} 
a_{l-2}, \dots, a_{p+1}, x_p, \al_{p-1}, \dots, \al_0 \\
b_{l-2}, \dots, b_{p+1}, y_p, \be_{p-1}, \dots, \be_0
\end{smallmatrix} \right]})
\right) \\
& := 
\left(
(a_l,c_l), \mcMhat\bigl((a_{l-1}, c_{l-1}), (b_{l-1}, d_{l-1}), f_{l-1 \left[ 
\begin{smallmatrix} 
(a_{l-2}, c_{l-2}), \dots, (a_{p+1}, c_{p+1}), x_p, \al_{p-1}, \dots, \al_0 \\
(b_{l-2}, d_{l-2}), \dots, (b_{p+1}, d_{p+1}), z_p, \be_{p-1}, \dots, \be_0
\end{smallmatrix} \right]}\bigr)
\right) .
\end{align*}


\subsubsection*{Case $l \in \N$ and $p=l-1$:}

There are no $a$'s, $b$'s, $c$'s and $d$'s in the `history index' which ends with $x_{l-1}$, $y_{l-1}$, $z_{l-1}$. We set 
\begin{align*}
& \left(
c_l, \mcMhat(y_{l-1}, z_{l-1}, f_{l-1 \left[ \begin{smallmatrix} \al_{l-2}, \dots, \al_0 \\ \be_{l-2}, \dots, \be_0 \end{smallmatrix} \right]})
\right) \\
& \quad \circ_{l-1}
\left(
a_l, \mcMhat(x_{l-1}, y_{l-1}, f_{l-1 \left[ \begin{smallmatrix} \al_{l-2}, \dots, \al_0 \\ \be_{l-2}, \dots, \be_0 \end{smallmatrix} \right]})
\right) \\
& := 
\left(
(a_l, c_l), \mcMhat(x_{l-1}, z_{l-1}, f_{l-1 \left[ \begin{smallmatrix} \al_{l-2}, \dots, \al_0 \\ \be_{l-2}, \dots, \be_0 \end{smallmatrix} \right]})
\right).
\end{align*}

Note that this construction depends on the choice of a family of Morse functions $F:=\{f_0, f_{1 \left[\begin{smallmatrix} x_0 \\ y_0 \end{smallmatrix} \right]}, \dots\}$ and metrics $G:= \{g_0, g_{1 \left[\begin{smallmatrix} x_0 \\ y_0 \end{smallmatrix} \right]}, \dots\}$.

\begin{Theorem}[Hohloch \cite{hohloch}]
\label{morsencategory}
The above defined $n$-globular set $X=\{X(l) \mid 0\leq l \leq n\}$ together with the above defined identity functions $\mathbf{1}$ and composites $\circ_p$ is an almost strict $n$-category $\mcX:=\mcX(F,G)$, called the \textbf{almost strict Morse $n$-category}.
\end{Theorem}


\section{Functors to the almost strict $n$-categories $\mcV$ and $\mcW$}

Let $M$ be a smooth compact manifold.
Denote by $\mcX=\mcX(F,G)$ the almost strict Morse $n$-category depending on Morse data $F=\left(f_{0, [ \begin{aligned} \dots \end{aligned} ]}, \dots \right)$ and $G=\left(g_{0, [ \begin{aligned} \dots \end{aligned} ]}, \dots \right)$ as defined in \refsubsectionmorsencat. In the present section, we introduce two almost strict $n$-categories $\mcV$ and $\mcW$ and provide $n$-functors $\mcF : \mcX  \to \mcV$ and $\mcF : \mcX  \to \mcW$. Since $\mcV$ and $\mcW$ are more accessible and easier to understand than $\mcX$ the functors help to understand the nature of $\mcX$. 

The idea is similar to representation theory of groups, where one studies homomorphisms from a given, often complicated group into an easier one, hoping to gain some knowledge of the complicated group via its image.


\subsection{The almost strict $n$-categories $\mcV$ and $\mcW$}

We define the $n$-globular set $V=\{V(l) \mid {0 \leq l \leq n} \}$ as follows. We set
\beqs
V(0):= \{ \R^{i_0} \mid i_0 \in \N_0 \}
\eeqs
and 
\begin{align*}
V(1):= \left\{ (\R^{i_1}, \Hom ( \R^{i_0}, \R^{j_0})) \left| \begin{aligned} & 0 \leq i_1 < i_0 - j_0, \\ & 0 \leq j_0 \leq i_0 \end{aligned} \right. \right\}
\end{align*}
and
\begin{align*}
V(2):= \left\{ \bigl(\R^{i_2}, \Hom(\R^{i_1}, \R^{j_1}), \Hom ( \R^{i_0}, \R^{j_0})\bigr) \left|
\begin{aligned} 
 & 0 \leq i_2 < i_1 - j_1, \\
 & 0 \leq j_1 \leq i_1 < i_0 - j_0, \\
 & 0 \leq j_0 \leq i_0
\end{aligned}
 \right. \right\}
\end{align*}
and generally for $n \geq l \geq 1$
\begin{align*}
V(l):= \left\{ \bigl(\R^{i_l}, \Hom(\R^{i_{l-1}}, \R^{j_{l-1}}), \dots,  \Hom ( \R^{i_0}, \R^{j_0})\bigr) \left|
\begin{aligned}
 & 0 \leq i_l <  i_{l-1} - j_{l-1} \\
 & 0 \leq j_{l-1} \leq i_{l-1} < i_{l-2} - j_{l-2} \\
 & \quad \quad \vdots \\
 & 0 \leq j_1 \leq i_1 < i_0 - j_0, \\
 & 0 \leq j_0 \leq i_0
\end{aligned}
 \right. \right\}.	
\end{align*}
We define the source functions $ s: V(l) \to V(l-1)$ and target functions $t: V(l) \to V(l-1)$ as follows. For $n \geq l \geq 2$ set
\begin{align*}
&  s\bigl(\R^{i_l}, \Hom(\R^{i_{l-1}}, \R^{j_{l-1}}), \dots,  \Hom ( \R^{i_0}, \R^{j_0})\bigr) \\
 & \qquad \qquad := \bigl(\R^{i_{l-1}}, \Hom(\R^{i_{l-2}}, \R^{j_{l-2}}), \dots,  \Hom ( \R^{i_0}, \R^{j_0})\bigr), \\
 &  t\bigl(\R^{i_l}, \Hom(\R^{i_{l-1}}, \R^{j_{l-1}}), \dots,  \Hom ( \R^{i_0}, \R^{j_0})\bigr) \\
 & \qquad \qquad := \bigl(\R^{j_{l-1}}, \Hom(\R^{i_{l-2}}, \R^{j_{l-2}}), \dots,  \Hom ( \R^{i_0}, \R^{j_0})\bigr), 
\end{align*}
and 
\begin{align*}
 s\bigl(\R^{i_1}, \Hom ( \R^{i_0}, \R^{j_0})\bigr) & : = \R^{i_0}, \\
 t\bigl(\R^{i_1}, \Hom ( \R^{i_0}, \R^{j_0})\bigr) & : = \R^{j_0}.
\end{align*}

\begin{Lemma}
\label{Vnglobular}
$V$ is an $n$-globular set with $s$ as source and $t$ as target function.
\end{Lemma}

\begin{proof}
We have to show $ s \circ s = s\circ t$ and $t \circ t = t \circ s$. We compute exemplarily
\begin{align*}
  & s\big(s\bigl(\R^{i_l}, \Hom(\R^{i_{l-1}}, \R^{j_{l-1}}), \dots,  \Hom ( \R^{i_0}, \R^{j_0})\bigr) \bigr) \\
  & \qquad \qquad = s\bigl(\R^{i_{l-1}}, \Hom(\R^{i_{l-2}}, \R^{j_{l-2}}), \dots,  \Hom ( \R^{i_0}, \R^{j_0})\bigr) \\
  & \qquad \qquad = \bigl(\R^{i_{l-2}}, \Hom(\R^{i_{l-3}}, \R^{j_{l-3}}), \dots,  \Hom ( \R^{i_0}, \R^{j_0})\bigr)  
\end{align*}
which coincides with
\begin{align*}
 & s\big(t\bigl(\R^{i_l}, \Hom(\R^{i_{l-1}}, \R^{j_{l-1}}), \dots,  \Hom ( \R^{i_0}, \R^{j_0})\bigr) \bigr) \\
 & \qquad \qquad = s\bigl(\R^{j_{l-1}}, \Hom(\R^{i_{l-2}}, \R^{j_{l-2}}), \dots,  \Hom ( \R^{i_0}, \R^{j_0})\bigr) \\
  & \qquad \qquad = \bigl(\R^{i_{l-2}}, \Hom(\R^{i_{l-3}}, \R^{j_{l-3}}), \dots,  \Hom ( \R^{i_0}, \R^{j_0})\bigr).
\end{align*}
\end{proof}

As identity functions $\mathbf 1 : V(l) \to V(l+1)$, we set on $V(0)$
\begin{align*}
 \mathbf 1 (\R^{i_0}) := \bigl( \R^0, \Hom(\R^{i_0}, \R^{i_0}) \bigr)= \bigl( \R^{i_0 - i_0}, \Hom(\R^{i_0}, \R^{i_0}) \bigr)
\end{align*}
and on $V(l)$ for $1 \leq l \leq n-1$
\begin{align*}
&  \mathbf 1 \bigl(\R^{i_l}, \Hom(\R^{i_{l-1}}, \R^{j_{l-1}}), \dots,  \Hom ( \R^{i_0}, \R^{j_0})\bigr) \\
 & \qquad \qquad := \bigl(\R^0, \Hom(\R^{i_l}, \R^{i_{l}}), \Hom(\R^{i_{l-1}}, \R^{j_{l-1}}) \dots,  \Hom ( \R^{i_0}, \R^{j_0})\bigr) \\
 & \qquad \qquad = \bigl(\R^{i_l - i_l},  \Hom(\R^{i_l}, \R^{i_{l}}), \Hom(\R^{i_{l-1}}, \R^{j_{l-1}}), \dots,  \Hom ( \R^{i_0}, \R^{j_0})\bigr). 
\end{align*}
In order to define the composite $\circ_p$ on $V(l) \x_p V(l)$ for $l > p \geq 0$ let us have a closer look at $V(l) \x_p V(l)$ first. 
A tuple $(R_l, Q_l) \in V(l) \x_p V(l)$ is in fact of the form
\begin{align*}
 Q_l & =\bigl(\R^{i_l}, \Hom(\R^{i_{l-1}}, \R^{j_{l-1}}), \dots,  \Hom ( \R^{i_0}, \R^{j_0}) \bigr)\\
&=\bigl(\R^{i_l}, \Hom(\R^{i_{l-1}}, \R^{j_{l-1}}), \dots,  \Hom(\R^{i_{p+1}}, \R^{j_{p+1}}),\\
&  \qquad \qquad \qquad \qquad \Hom(\R^{u_p}, \R^{v_p}), \Hom(\R^{\rho_{p-1}}, \R^{\si_{l-1}}), \dots,  \Hom ( \R^{\rho_0}, \R^{\si_0})\bigr)
\end{align*}
and
\begin{align*}
R_l & = \bigl(\R^{\mu_l}, \Hom(\R^{\mu_{l-1}}, \R^{\nu_{l-1}}), \dots,  \Hom ( \R^{\mu_0}, \R^{\nu_0}) \bigr) \\
& =\bigl(\R^{\mu_l}, \Hom(\R^{\mu_{l-1}}, \R^{\nu_{l-1}}), \dots ,\Hom(\R^{\mu_{p+1}}, \R^{\nu_{p+1}}),   \\
& \qquad \qquad \qquad \qquad \Hom(\R^{v_p}, \R^{w_p}), \Hom(\R^{\rho_{p-1}}, \R^{\si_{l-1}}), \dots,  \Hom ( \R^{\rho_0}, \R^{\si_0})\bigr)
\end{align*}
where the indices $i_k=\mu_k=:\rho_k$ and $j_k = \nu_k =: \si_k$ coincide for $0 \leq k \leq p-1$. For $k=p$, we have the matching condition $i_p=:u_p$, $j_p=\mu_p =:v_p$ and $\nu_p=:w_p$. For $p+1 \leq k \leq l$, there are no conditions on $i_k$, $j_k$, $\mu_k$ and $\nu_k$.
Using these conventions for $(R_l, Q_l):= V(l) \x_p V(l)$, we define
\begin{align*}
 & R_l \circ_p Q_l \\
 & \quad  := \bigl(\R^{i_l + \mu_l}, \Hom(\R^{i_{l-1}+ \mu_{l-1}}, \R^{j_{l-1} + \nu_{l-1}}), \dots,  \Hom(\R^{i_{p+1} + \mu_{p+1}}, \R^{j_{p+1} + \nu_{p+1}}),   \\
& \qquad \qquad \qquad\qquad \qquad \qquad \qquad  \ \ \Hom(\R^{u_p}, \R^{w_p}), \Hom(\R^{\rho_{p-1}}, \R^{\si_{l-1}}) \Hom ( \R^{\rho_0}, \R^{\si_0})\bigr) 
 \end{align*}
where we canonically identify $\R^{i_l + \mu_l} \simeq \R^{i_l} \x \R^{\mu_l}$ etc.

\begin{Theorem}
\label{homncat}
The $n$-globular set $V=\{V(l) \mid 0 \leq l \leq n \}$ together with the above defined identity functions $\mathbf 1$ and the composite $\circ_p$ yields an almost strict $n$-category $\mcV$.
\end{Theorem}

Before we turn to the proof of \refhomncat, let us define a second almost strict $n$-category. For $k \in \N_0$, abbreviate $\N_0^k:=(\N_0)^k$ and define $W=\{W(l) \mid 0 \leq l \leq n\}$ via
\begin{align*}
 W(0):= \N_0
\end{align*}
and
\begin{align*}
 W(1):= \left\{ \left( i_1, \left[ \begin{aligned} i_0 \\ j_0 \end{aligned} \right] \right)  \in \N_0 \x \N_0^2  \left| \begin{aligned} & 0 \leq i_1 < i_0 - j_0, \\ & 0 \leq j_0 \leq i_0 \end{aligned} \right. \right\}
\end{align*}
and 
\begin{align*}
W(2):= \left\{ \left( i_2, \left[ \begin{aligned} i_1, i_0 \\ j_1, j_0 \end{aligned} \right] \right)  \in \N_0 \x (\N_0^2)^2    \left|
\begin{aligned} 
 & 0 \leq i_2 < i_1 - j_1, \\
 & 0 \leq j_1 \leq i_1 < i_0 - j_0, \\
 & 0 \leq j_0 \leq i_0
\end{aligned}
 \right. \right\}
\end{align*}
and generally for $n \geq l \geq 1$
\begin{align*}
W(l):= \left\{\left( i_l, \left[ \begin{aligned} i_{l-1}, \dots, i_0 \\ j_{l-1}, \dots,  j_0 \end{aligned} \right] \right)  \in \N_0 \x (\N_0^2)^l       \left|
\begin{aligned}
 & 0 \leq i_l <  i_{l-1} - j_{l-1} \\
 & 0 \leq j_{l-1} \leq i_{l-1} < i_{l-2} - j_{l-2} \\
 & \quad \quad \vdots \\
 & 0 \leq j_1 \leq i_1 < i_0 - j_0, \\
 & 0 \leq j_0 \leq i_0
\end{aligned}
 \right. \right\}.	
\end{align*}

Let us make the following important observation.

\begin{Remark}
 \label{VW}
Although $V$ and $W$ are clearly different sets, we can abbreviate elements in $V(l)$ by means of elements in $W(l)$ via identifying
\begin{align*}
 \bigl(\R^{i_l}, \Hom(\R^{i_{l-1}}, \R^{j_{l-1}}), \dots,  \Hom ( \R^{i_0}, \R^{j_0})\bigr)
 \mathrel{\widehat{=}}
 \left( i_l, \left[ \begin{aligned}
                     & i_{l-1}, \dots, i_0 \\
                     & j_{l-1}, \dots, j_0
                    \end{aligned} \right] \right)          
\end{align*}
which simplifies the notation considerably. Since the dimensions of the vector spaces in $V$ satisfy the same constraints as the integers in $W$ we can even use this short notation in proofs, keeping in mind the different canonical isomorphims of $V$ and $W$.
\end{Remark}

For $1 \leq l\leq n$, in analogy to $V$, we have the source and target functions $s: W(l) \to W(l-1)$ and $t: W(l) \to W(l-1)$
\begin{align*}
 s \left( i_l, \left[ \begin{aligned} i_{l-1}, \dots, i_0 \\ j_{l-1}, \dots,  j_0 \end{aligned} \right] \right) 
 & := \left( i_{l-1}, \left[ \begin{aligned} i_{l-2}, \dots, i_0 \\ j_{l-2}, \dots,  j_0 \end{aligned} \right] \right), \\
  t \left( i_l, \left[ \begin{aligned} i_{l-1}, \dots, i_0 \\ j_{l-1}, \dots,  j_0 \end{aligned} \right] \right) 
 & := \left( j_{l-1}, \left[ \begin{aligned} i_{l-2}, \dots, i_0 \\ j_{l-2}, \dots,  j_0 \end{aligned} \right] \right)
\end{align*}
for $l>1$ and 
\begin{align*}
 s\left( i_1, \left[ \begin{aligned} i_0 \\ j_0 \end{aligned} \right] \right):= i_0 ,\qquad \quad
 t\left( i_1, \left[ \begin{aligned} i_0 \\ j_0 \end{aligned} \right] \right):= j_0
\end{align*}
for $l=1$.

\begin{Lemma}
\label{Wnglobular}
 $W$ is an $n$-globular set with $s$ as source and $t$ as target function.
\end{Lemma}

\begin{proof}
Using \refVW, the claim follows from \refVnglobular.
\end{proof}

Similar to $V$, as identity functions $\mathbf 1 : W(l) \to W(l+1)$, we set on $W(0)$
$$
\mathbf 1 (i_0):= \left( 0, \left[ \begin{aligned} i_0 \\ i_0 \end{aligned} \right] \right)
$$
and on $W(l)$ with $l>0$
\begin{align*}
 \mathbf 1 \left( i_l, \left[ \begin{aligned} i_{l-1}, \dots, i_0 \\ j_{l-1}, \dots,  j_0 \end{aligned} \right] \right) 
 & := \left( 0 , \left[ \begin{aligned} i_{l}, i_{l-1} \dots, i_0 \\ i_l, j_{l-1}, \dots,  j_0 \end{aligned} \right] \right).
\end{align*}
Let us introduce the composite for $W$. A tuple $(R_l, Q_l) \in W(l) \x_p W(l)$ is of the form
\begin{align*}
 Q_l  & = \left( i_l, \left[ \begin{aligned}
                     & i_{l-1}, \dots,i_{p+1}, u_p, \rho_{p-1}, \dots,  \rho_0 \\
                     & j_{l-1}, \dots, j_{p+1}, v_p, \si_{p-1}, \dots,  \si_0
                    \end{aligned} \right] \right) ,  \\
 R_l  & = \left( \mu_l, \left[ \begin{aligned}
                     & \mu_{l-1}, \dots,\mu_{p+1}, v_p, \rho_{p-1}, \dots,  \rho_0 \\
                     & \nu_{l-1}, \dots, \nu_{p+1}, w_p, \si_{p-1}, \dots,  \si_0
                    \end{aligned} \right] \right)                    
\end{align*}
where the indices are defined as in the case of the composite on $V$. We define
\begin{align*}
 & R_l \circ_p  Q_l\\
 & =              \left( \mu_l, \left[ \begin{aligned}
                     & \mu_{l-1}, \dots,\mu_{p+1}, v_p, \rho_{p-1}, \dots,  \rho_0 \\
                     & \nu_{l-1}, \dots, \nu_{p+1}, w_p, \si_{p-1}, \dots,  \si_0
                    \end{aligned} \right] \right)  
                     \circ_p
                     \left( i_l, \left[ \begin{aligned}
                     & i_{l-1}, \dots,i_{p+1}, u_p, \rho_{p-1}, \dots,  \rho_0 \\
                     & j_{l-1}, \dots, j_{p+1}, v_p, \si_{p-1}, \dots,  \si_0
                    \end{aligned} \right] \right) \\
               & : =   \left( (i_l + \mu_l) ,\left[ \begin{aligned}
                     & (i_{l-1} + \mu_{l-1}), \dots,(i_{p+1} + \mu_{p+1}), u_p, \rho_{p-1}, \dots,  \rho_0 \\
                     & (j_{l-1} + \nu_{l-1}), \dots, (j_{p+1}+  \nu_{p+1}), w_p, \si_{p-1}, \dots,  \si_0
                    \end{aligned} \right] \right)                
\end{align*}

\begin{Theorem}
\label{Wncat}
The $n$-globular set $W=\{W(l) \mid 0 \leq l \leq n \}$ together with the above defined identity functions $\mathbf 1$ and the composite $\circ_p$ yields an almost strict $n$-category denoted by $\mcW$.
\end{Theorem}

\begin{proof}[Proof of \refhomncat\ and \refWncat]
Using \refVW, we can prove \refhomncat\ and \refWncat\ simultanously if we point out the different underlying canonical isomorphisms accordingly.

{\em (a) Source and target functions of composites:} Let $(R_l, Q_l) \in V(l) \x_p V(l)$. We have to show that $s(R_l \circ_p Q_l) = s(Q_l)$ and $t(R_l \circ_p Q_l) = t(R_l)$ for $p=l-1$. We calculate for $l \geq 1$
\begin{align*}
 & s(R_l \circ_p Q_l) \\
 & =             s \left( \left( \mu_l, \left[ \begin{aligned}
                     &  v_{l-1}, \rho_{l-2}, \dots,  \rho_0 \\
                     &  w_{l-1}, \si_{l-2}, \dots,  \si_0
                    \end{aligned} \right] \right)  
                     \circ_p
                     \left( i_l, \left[ \begin{aligned}
                     &  u_{l-1}, \rho_{l-2}, \dots,  \rho_0 \\
                     &  v_{l-1}, \si_{l-2}, \dots,  \si_0
                    \end{aligned} \right] \right) \right)\\
&  =  s \left( (i_l + \mu_l) ,\left[ \begin{aligned}
                     &  u_{l-1}, \rho_{l-2}, \dots,  \rho_0 \\
                     &  w_{l-1}, \si_{l-2}, \dots,  \si_0
                    \end{aligned} \right] \right)  \\
& = \left( u_{l-1} ,\left[ \begin{aligned}
                     &   \rho_{l-2}, \dots,  \rho_0 \\
                     &   \si_{l-2}, \dots,  \si_0
                    \end{aligned} \right] \right)  \\  
&= s  \left( i_l, \left[ \begin{aligned}
                     &  u_{l-1}, \rho_{l-2}, \dots,  \rho_0 \\
                     &  v_{l-1}, \si_{l-2}, \dots,  \si_0
                    \end{aligned} \right] \right)  \\
& = s(Q_l).
\end{align*}
For $l=1$ we find 
\begin{align*}
 s \left(\left( \mu_1, \left[ \begin{aligned} v_0 \\ w_0 \end{aligned} \right]\right) \circ_0 \left(i_1,  \left[ \begin{aligned} u_0 \\ v_0 \end{aligned} \right]  \right) \right) = s\left( i_1 + \mu_1,  \left[ \begin{aligned} u_0 \\ w_0 \end{aligned} \right]\right)=u_0= s \left(i_1,  \left[ \begin{aligned} u_0 \\ v_0 \end{aligned} \right]  \right)
\end{align*}
and similar calculations prove the claim for the target function $t$. Now we address the case $0 \leq p \leq l-2$. For $(R_l, Q_l) \in V(l) \x_p V(l)$, we have to show that $s(R_l \circ_p Q_l) = s(R_l) \circ_p s(Q_l)$ and $t(R_l \circ_p Q_l) = t(R_l) \circ_p t(Q_l)$. We compute
\begin{align*}
&  s(R_l \circ_p Q_l) \\
 & = s \left(\left( \mu_l, \left[ \begin{smallmatrix}
                     \mu_{l-1}, \dots,\mu_{p+1}, v_p, \rho_{p-1}, \dots,  \rho_0 \\
                      \nu_{l-1}, \dots, \nu_{p+1}, w_p, \si_{p-1}, \dots,  \si_0
                    \end{smallmatrix} \right] \right)  
                     \circ_p
                     \left( i_l, \left[ \begin{smallmatrix}
                     & i_{l-1}, \dots,i_{p+1}, u_p, \rho_{p-1}, \dots,  \rho_0 \\
                     & j_{l-1}, \dots, j_{p+1}, v_p, \si_{p-1}, \dots,  \si_0
                    \end{smallmatrix} \right] \right) \right) \\
&  =  s \left( (i_l + \mu_l) ,\left[ \begin{smallmatrix}
                     & (i_{l-1} + \mu_{l-1}), \dots,(i_{p+1} + \mu_{p+1}), u_p, \rho_{p-1}, \dots,  \rho_0 \\
                     & (j_{l-1} + \nu_{l-1}), \dots, (j_{p+1}+  \nu_{p+1}), w_p, \si_{p-1}, \dots,  \si_0
                    \end{smallmatrix} \right] \right) \\
&  =  \left( (i_{l-1} + \mu_{l-1}) ,\left[ \begin{smallmatrix}
                     & (i_{l-2} + \mu_{l-2}), \dots,(i_{p+1} + \mu_{p+1}), u_p, \rho_{p-1}, \dots,  \rho_0 \\
                     & (j_{l-2} + \nu_{l-2}), \dots, (j_{p+1}+  \nu_{p+1}), w_p, \si_{p-1}, \dots,  \si_0
                    \end{smallmatrix} \right] \right) \\
& = \left( \mu_{l-1}, \left[ \begin{smallmatrix}
                     & \mu_{l-2}, \dots,\mu_{p+1}, v_p, \rho_{p-1}, \dots,  \rho_0 \\
                     & \nu_{l-2}, \dots, \nu_{p+1}, w_p, \si_{p-1}, \dots,  \si_0
                    \end{smallmatrix} \right] \right)  
                     \circ_p
                     \left( i_{l-1}, \left[ \begin{smallmatrix}
                     & i_{l-2}, \dots,i_{p+1}, u_p, \rho_{p-1}, \dots,  \rho_0 \\
                     & j_{l-2}, \dots, j_{p+1}, v_p, \si_{p-1}, \dots,  \si_0
                    \end{smallmatrix} \right] \right) \\
& = s \left( \mu_l, \left[ \begin{smallmatrix}
                     & \mu_{l-1}, \dots,\mu_{p+1}, v_p, \rho_{p-1}, \dots,  \rho_0 \\
                     & \nu_{l-1}, \dots, \nu_{p+1}, w_p, \si_{p-1}, \dots,  \si_0
                    \end{smallmatrix} \right] \right)  
                     \circ_p
                     s \left( i_l, \left[ \begin{smallmatrix}
                     & i_{l-1}, \dots,i_{p+1}, u_p, \rho_{p-1}, \dots,  \rho_0 \\
                     & j_{l-1}, \dots, j_{p+1}, v_p, \si_{p-1}, \dots,  \si_0
                    \end{smallmatrix} \right] \right) \\
 & = s(R_l) \circ_p s(Q_l).
\end{align*}
The claim follows similarly for the target function.

\vsp


{\em (b) Sources and targets of identities:} We show $s(\mathbf 1 (Q_l))=Q_l = t(\mathbf 1 (Q_l))$ via
\begin{align*}
  s(\mathbf 1 (Q_l)) 
 & = s \left( \mathbf 1 \left( i_l, \left[ \begin{aligned}
                     & i_{l-1}, \dots,i_{0} \\
                     & j_{l-1}, \dots, j_{0}
                    \end{aligned} \right] \right) \right) 
 = s  \left( 0, \left[ \begin{aligned}
                     & i_l, i_{l-1}, \dots,i_{0} \\
                     & i_l, j_{l-1}, \dots, j_{0}
                    \end{aligned} \right] \right) \\
 & =  \left( i_l, \left[ \begin{aligned}
                     & i_{l-1}, \dots,i_{0} \\
                     & j_{l-1}, \dots, j_{0}
                    \end{aligned} \right] \right) 
 = t  \left( 0, \left[ \begin{aligned}
                     & i_l, i_{l-1}, \dots,i_{0} \\
                     & i_l, j_{l-1}, \dots, j_{0}
                    \end{aligned} \right] \right) 
 =t(\mathbf 1 (Q_l)).
\end{align*}

\vsp


{\em (c) Associativity of the composite:}
Let $0 \leq p < l \leq n$ and $(R_l, Q_l), (S_l, R_l) \in V(l) \x_p V(l)$ and show $(S_l \circ_p R_l) \circ_p Q_l = S_l \circ_p (R_l \circ_p Q_l)$. We write
\begin{align*}
Q_l & = \left( i_l, \left[ \begin{aligned}
                     & i_{l-1}, \dots,i_{p+1}, u_p, \rho_{p-1}, \dots,  \rho_0 \\
                     & j_{l-1}, \dots, j_{p+1}, v_p, \si_{p-1}, \dots,  \si_0
                    \end{aligned} \right] \right), \\
R_l & =\left( \ka_l, \left[ \begin{aligned}
                     & \ka_{l-1}, \dots,\ka_{p+1}, v_p, \rho_{p-1}, \dots,  \rho_0 \\
                     & \lam_{l-1}, \dots, \lam_{p+1}, w_p, \si_{p-1}, \dots,  \si_0
                    \end{aligned} \right] \right) ,  \\                    
S_l & =\left( \mu_l, \left[ \begin{aligned}
                     & \mu_{l-1}, \dots,\mu_{p+1}, w_p, \rho_{p-1}, \dots,  \rho_0 \\
                     & \nu_{l-1}, \dots, \nu_{p+1}, x_p, \si_{p-1}, \dots,  \si_0
                    \end{aligned} \right] \right)  
\end{align*}
and compute
\begin{align*}
 & (S_l \circ_p R_l) \circ_p Q_l \\
 & = \left((\ka_l + \lam_l) + i_l, \left[ \begin{aligned}
                     & (\ka_{l-1} + \mu_{l-1} ) + i_{l-1}, \dots, (\ka_{p+1} + \mu_{p+1}) + i_{p+1}, u_p, \rho_{p-1}, \dots,  \rho_0 \\
                     & (\lam_{l-1} + \nu_{l-1}) + j_{l-1}, \dots, (\lam_{p+1} + \nu_{p+1}) + j_{p+1}, x_p, \si_{p-1}, \dots,  \si_0
                    \end{aligned} \right] \right) \\
& = \left(\ka_l + (\lam_l + i_l), \left[ \begin{aligned}
                     & \ka_{l-1} + (\mu_{l-1}  + i_{l-1}), \dots, \ka_{p+1} + (\mu_{p+1} + i_{p+1}), u_p, \rho_{p-1}, \dots,  \rho_0 \\
                     & \lam_{l-1} + (\nu_{l-1} + j_{l-1}), \dots, \lam_{p+1} + (\nu_{p+1} + j_{p+1}), x_p, \si_{p-1}, \dots,  \si_0
                    \end{aligned} \right] \right) \\                   
& = S_l \circ_p (R_l \circ_p Q_l).                   
\end{align*}
For $V$, we considered the cartesian product and the composition of matrices as associative using canonical isomorphisms. In case of $W$, we need the associativity of the addition of integers.

\vspace{3mm}


{\em (d) Identities:}
 For $0 \leq p < l \leq n$ and $Q_l \in V(l)$, we have to show 
\beqs
\mathbf 1^{l-p}(t^{l-p}(Q_l)) \circ_p Q_l = Q_l = Q_l \circ_p \mathbf 1^{l-p}(s^{l-p}(Q_l)).
\eeqs
For 
\begin{align*}
 Q_l & = \left( i_l, \left[ \begin{aligned}
                     & i_{l-1}, \dots, i_0 \\
                     & j_{l-1}, \dots,j_0
                    \end{aligned} \right] \right)
\end{align*}
compute
\begin{align*}
 & \mathbf 1^{l-p}(t^{l-p}(Q_l))  \circ_p Q_l \\
 & = \mathbf 1^{l-p} \left( t^{l-p} 
 \left( i_l, \left[ \begin{aligned}
                     & i_{l-1}, \dots, i_0 \\
                     & j_{l-1}, \dots,j_0
                    \end{aligned} \right] \right)
                    \right)  
                    \circ_p
 \left( i_l, \left[ \begin{aligned}
                     & i_{l-1}, \dots, i_0 \\
                     & j_{l-1}, \dots,j_0
                    \end{aligned} \right] \right) \\
& =  \mathbf 1^{l-p} 
		  \left( j_p, \left[ \begin{aligned}
                     & i_{p-1}, \dots, i_0 \\
                     & j_{p-1}, \dots,j_0
                    \end{aligned} \right] \right) 
                    \circ_p
 \left( i_l, \left[ \begin{aligned}
                     & i_{l-1}, \dots, i_0 \\
                     & j_{l-1}, \dots,j_0
                    \end{aligned} \right] \right) \\
& =  \left( 0, \left[ \begin{aligned}
                     & 0, \dots,0, j_p, i_{p-1}, \dots, i_0 \\
                     & 0, \dots,0, j_p, j_{p-1}, \dots,j_0
                    \end{aligned} \right] \right) 
                    \circ_p
 \left( i_l, \left[ \begin{aligned}
                     & i_{l-1}, \dots, i_0 \\
                     & j_{l-1}, \dots,j_0
                    \end{aligned} \right] \right) \\ 
& = \left(i_l , \left[ \begin{aligned}
                     & i_{l-1} , \dots, i_{p+1} , i_p, i_{p-1}, \dots, i_0 \\
                     & j_{l-1}, \dots, j_{p+1} , j_p, j_{p-1}, \dots,j_0
                    \end{aligned} \right] \right) \\
& = Q_l.                    
\end{align*}
Note that for $V$, in the above calculation, we used canonical isomorphims like $\R^0 \x \R^{i_l} \simeq \R^{i_l}$ a couple of times. The assertion $Q_l = Q_l \circ_p \mathbf 1^{l-p}(s^{l-p}(Q_l))$ follows analogously.

\vsp


{\em (e) Binary interchange:}
Given $0 \leq q <p <l \leq n$ and $(R_l, Q_l)$, $(T_l, S_l) \in V(l)\x_p V(l)$ and $(T_l, R_l)$, $(S_l, Q_l) \in V(l)\x_q V(l)$, we need to show 
$$(T_l \circ_p S_l) \circ_q (R_l \circ_p Q_l) = (T_l \circ_q R_l) \circ_p (S_l \circ _p Q_l).$$
$Q_l$, $R_l$, $S_l$ and $T_l$ are of the form 
\begin{align*}
Q_l & = \left( i_l, \left[ \begin{aligned}
                     & i_{l-1}, \dots,i_{p+1}, u_p, \rho_{p-1}, \dots,  \rho_0 \\
                     & j_{l-1}, \dots, j_{p+1}, v_p, \si_{p-1}, \dots,  \si_0
                    \end{aligned} \right] \right), \\
R_l & =\left( \ka_l, \left[ \begin{aligned}
                     & \ka_{l-1}, \dots,\ka_{p+1}, v_p, \rho_{p-1}, \dots,  \rho_0 \\
                     & \lam_{l-1}, \dots, \lam_{p+1}, w_p, \si_{p-1}, \dots,  \si_0
                    \end{aligned} \right] \right) ,  \\                    
S_l & =\left( \mu_l, \left[ \begin{aligned}
                     & \mu_{l-1}, \dots,\mu_{p+1}, \ubar_p, \xi_{p-1} \dots ,\xi_{q+1}, \si_q, \rho_{q-1}, \dots,  \rho_0 \\
                     & \nu_{l-1}, \dots, \nu_{p+1}, \vbar_p, \tau_{p-1}, \dots, \tau_{q+1}, \tau_q , \si_{q-1},  \dots , \si_0
                    \end{aligned} \right] \right), \\
T_l & =\left( \phi_l, \left[ \begin{aligned}
                     & \phi_{l-1}, \dots,\phi_{p+1}, \vbar_p, \xi_{p-1} \dots ,\xi_{q+1}, \si_q, \rho_{q-1}, \dots,  \rho_0 \\
                     & \psi_{l-1}, \dots, \psi_{p+1}, \wbar_p, \tau_{p-1}, \dots, \tau_{q+1}, \tau_q , \si_{q-1},  \dots , \si_0
                    \end{aligned} \right] \right).                   
\end{align*}
We compute
\begin{align*}
 & (T_l \circ_p S_l) \circ_q (R_l \circ_p Q_l) \\
 & = \left( (\mu_l + \phi_l) , \left[ \begin{aligned}
                     & (\mu_{l-1} + \phi_{l-1}), \dots,(\mu_{p+1} + \phi_{p+1}), \ubar_p, \xi_{p-1} \dots ,\xi_{q+1}, \si_q, \rho_{q-1}, \dots,  \rho_0 \\
                     & (\nu_{l-1} + \psi_{l-1}) , \dots, (\nu_{p+1} + \psi_{p+1}), \wbar_p, \tau_{p-1}, \dots, \tau_{q+1}, \tau_q , \si_{q-1},  \dots  \si_0
                    \end{aligned} \right] \right) \\
 & \quad \quad \circ_q
  \left( (i_l + k_l) , \left[ \begin{aligned}
                     & (i_{l-1} + \ka_{l-1}), \dots,(i_{p+1} + \ka_{p+1}), u_p, \rho_{p-1}, \dots,  \rho_0 \\
                     & (j_{l-1} + \lam_{l-1}), \dots,( j_{p+1} + \lam_{p+1}), w_p, \si_{p-1}, \dots,  \si_0
                    \end{aligned} \right] \right) \\
& = \left( (i_l + k_l) + (\mu_l + \phi_l), \Gamma \right)             
\end{align*}
where
\begin{align*}
 & \Gamma:= \left[
    \begin{smallmatrix}
     \bigl((i_{l-1} + \ka_{l-1}) + (\mu_{l-1} + \phi_{l-1})\bigr), \dots, \bigl((i_{p+1} + \ka_{p+1}) + (\mu_{p+1} + \phi_{p+1})\bigr), \\
      \bigl( (j_{l-1} + \lam_{l-1}) + (\nu_{l-1} + \psi_{l-1})\bigr) , \dots, \bigl( ( j_{p+1} + \lam_{p+1})+ (\nu_{p+1} + \psi_{p+1}) \bigr), 
    \end{smallmatrix}
      \right. \\
& \\
& \qquad \qquad \qquad   \qquad \qquad  \qquad \qquad \qquad  
    \left.
      \begin{smallmatrix}
      (u_p + \ubar_p), (\rho_{p-1} + \xi_{p-1}) , \dots, (\rho_{q+1} + \xi_{q+1}), \rho_q, \rho_{q-1}, \dots,   \rho_0 \\
       (w_p + \wbar_p), (\si_{p-1} + \tau_{p-1}), \dots, (\si_{q+1} + \tau_{q+1}), \tau_q , \si_{q-1},  \dots  ,\si_0
    \end{smallmatrix}
    \right] .
\end{align*}
On the other hand, we have
\begin{align*}
 & (T_l \circ_q R_l) \circ_p (S_l \circ _p Q_l) \\
 & = \left( (\ka_l + \phi_l), \left[ \begin{smallmatrix}
                      (\ka_{l-1} + \phi_{l-1}), \dots,(\ka_{p+1} + \phi_{p+1}), (v_p + \vbar_p), (\rho_{p-1} + \xi_{p-1}) \dots ,(\rho_{q+1} +\xi_{q+1}), \rho_q, \rho_{q-1}, \dots,  \rho_0 \\
                      (\lam_{l-1} + \psi_{l-1}), \dots, (\lam_{p+1} + \psi_{p+1}),(w_p +  \wbar_p), (\si_{p-1} + \tau_{p-1}), \dots, (\si_{q+1} + \tau_{q+1}), \tau_q , \si_{q-1},  \dots , \si_0
                    \end{smallmatrix} \right] \right) \\  
& \quad \circ_q 
  \left( (i_l + \mu_l), \left[ \begin{smallmatrix}
                      (i_{l-1} + \mu_{l-1}), \dots,(i_{p+1} + \mu_{p+1}),(u_p + \ubar_p),(\rho_{p-1} +  \xi_{p-1}) \dots ,(\rho_{q+1} + \xi_{q+1}), \rho_q, \rho_{q-1}, \dots,  \rho_0 \\
                      (j_{l-1} + \nu_{l-1}), \dots, (j_{p+1} + \nu_{p+1}),(v_p +  \vbar_p),(\si_{p-1} +  \tau_{p-1}), \dots,(\si_{q+1} + \tau_{q+1}), \tau_q , \si_{q-1},  \dots , \si_0
                    \end{smallmatrix} \right] \right) \\ 
& = \left( (i_l + \mu_l) + (\ka_l + \phi_l), \De \right)
\end{align*}
where
\begin{align*}
& \De:= \left[ \begin{smallmatrix}
               \bigl((i_{l-1} + \mu_{l-1}) + (\ka_{l-1} + \phi_{l-1}) \bigr), \dots,\bigl( (i_{p+1} + \mu_{p+1}) + (\ka_{p+1} + \phi_{p+1}) \bigr), \\
                \bigl( (j_{l-1} + \nu_{l-1}) + (\lam_{l-1} + \psi_{l-1}) \bigr), \dots, \bigl( (j_{p+1} + \nu_{p+1}) + (\lam_{p+1} + \psi_{p+1}) \bigr) ,
              \end{smallmatrix}
	\right. \\
& \\	
& \quad \qquad \qquad \qquad \qquad \qquad \qquad	\left. \begin{smallmatrix}
               (u_p + \ubar_p),(\rho_{p-1} +  \xi_{p-1}) \dots ,(\rho_{q+1} + \xi_{q+1}), \rho_q, \rho_{q-1}, \dots,  \rho_0 \\
               (w_p +  \wbar_p),(\si_{p-1} +  \tau_{p-1}), \dots,(\si_{q+1} + \tau_{q+1}), \tau_q , \si_{q-1},  \dots , \si_0
              \end{smallmatrix}
	\right]	.
\end{align*}
Up to canonical isomorphisms of the type 
$$ \R^{(i_l + \mu_l) } \x \R^{ (\ka_l + \phi_l) }\simeq \R^{i_l + \mu_l + \ka_l + \phi_l} \simeq \R^{(i_l + \ka_l) } \x \R^{ (\mu_l + \phi_l) },$$ 
we proved the claim for $V$. For $W$, we need commutativity and associativity of the addition of integers.

\vsp


{\em (f) Nullary interchange:} For $0 \leq p <l <n$ and $(R_l, Q_l) \in V(l) \x _p V(l)$, we need to show $\mathbf 1_{R_l} \circ_p \mathbf 1_{Q_l} = \mathbf 1_{R_l \circ_p Q_l}$. Consider
\begin{align*}
Q_l & = \left( i_l, \left[ \begin{aligned}
                     & i_{l-1}, \dots,i_{p+1}, u_p, \rho_{p-1}, \dots,  \rho_0 \\
                     & j_{l-1}, \dots, j_{p+1}, v_p, \si_{p-1}, \dots,  \si_0
                    \end{aligned} \right] \right), \\
R_l & =\left( \ka_l, \left[ \begin{aligned}
                     & \ka_{l-1}, \dots,\ka_{p+1}, v_p, \rho_{p-1}, \dots,  \rho_0 \\
                     & \lam_{l-1}, \dots, \lam_{p+1}, w_p, \si_{p-1}, \dots,  \si_0
                    \end{aligned} \right] \right)            
\end{align*}
and compute
\begin{align*}
& \mathbf 1_{R_l} \circ_p \mathbf 1_{Q_l} \\
& = \left( 0, \left[ \begin{smallmatrix}
                      \ka_l, \ka_{l-1}, \dots,\ka_{p+1}, v_p, \rho_{p-1}, \dots,  \rho_0 \\
                      \ka_l, \lam_{l-1}, \dots, \lam_{p+1}, w_p, \si_{p-1}, \dots,  \si_0
                    \end{smallmatrix} \right] \right)   
 \circ_p
    \left( 0 , \left[ \begin{smallmatrix}
                     i_l, i_{l-1}, \dots,i_{p+1}, u_p, \rho_{p-1}, \dots,  \rho_0 \\
                      i_l , j_{l-1}, \dots, j_{p+1}, v_p, \si_{p-1}, \dots,  \si_0
                    \end{smallmatrix} \right] \right) \\
& = \left( 0 , \left[ \begin{smallmatrix}
                     (i_l + \ka_l) , (i_{l-1} + \ka_{l-1}) , \dots, (i_{p+1} + \ka_{p+1}) , u_p, \rho_{p-1}, \dots,  \rho_0 \\
                      (i_l + \ka_l) , (j_{l-1} + \lam_{l-1}) , \dots, ( j_{p+1} + \lam_{p+1}), w_p, \si_{p-1}, \dots,  \si_0
                    \end{smallmatrix} \right] \right) \\
& = \mathbf 1_{R_l \circ_p Q_l}
\end{align*}


{\em (g) The conditions on the indices in $V$ and $W$:}
Either one can show directly that the composite and the identity functions preserve the index requirements in the definition of $V$ and $W$ or one consults (d) in the proof of \refhomfunctor. This finishes the proof of \refhomncat.
\end{proof}


\subsection{The functors $\mcF : \mcX \to \mcV$ and $\mcG: \mcX \to \mcW$}

Let us now define the $n$-functor $\mcF: \mcX \to \mcV$. The functor will preserve the levels of the $n$-globular sets $X=\{X(l) \mid 0 \leq l \leq n\}$ and $V=\{V(l) \mid  0 \leq l \leq n\}$. We assume the notations and setting of \refsubsectionmorsencat. Recall that the Morse index of a critical point $x$ is denoted by $\Ind(x)$. We set
\begin{align*}
 \mcF: X(0)= \Crit(f_0) \to V(0) = \{ \R^{i_0} \mid i_0 \in \N_0\}, \quad x_0 \mapsto \R^{\Ind(x_0)}
\end{align*}
and 
\begin{gather*}
 \mcF: X(1) = \{ (x_1, \mcMhat(x_0, y_0, f_0)) \mid \dots\} \to V(1) = \{ (\R^{i_1} , \Hom(\R^{i_0}, \R^{j_0})) \mid \dots\} \\
 (x_1, \mcMhat(x_0, y_0, f_0))\quad  \mapsto \quad (\R^{\Ind(x_1)} , \Hom(\R^{\Ind(x_0)}, \R^{\Ind(y_0)}))
\end{gather*}
and generally for $1 \leq l \leq n$ using short notation
\begin{gather*}
\mcF: X(l) \to V(l) \\
\left(a_l, \mcMhat\bigl(a_{l-1}, b_{l-1}, f_{l-1
\left[
\begin{smallmatrix}
{a_{l-2}, \dots, a_0} \\
{b_{l-2}, \dots, b_0}
\end{smallmatrix}
\right]
}\bigr)\right)
\mapsto
\left(
\Ind(a_l), \left[ \begin{smallmatrix}
            \Ind(a_{l-1}), \dots, \Ind(a_0) \\
            \Ind(b_{l-1}), \dots, \Ind(b_0)
           \end{smallmatrix}
\right]\right).
\end{gather*}

\begin{Theorem}
\label{homfunctor}
$\mcF$ is an almost strict $n$-functor from $\mcX$ to $\mcV$.
\end{Theorem}

Before we prove \refhomfunctor\ we define an $n$-functor $\mcG$ from $\mcX $ to $\mcW$. We set
\begin{align*}
 \mcG: X(0)= \Crit(f_0) \to W(0) = \N_0, \quad x_0 \mapsto \Ind(x_0)
\end{align*}
and 
\begin{gather*}
 \mcG: X(1) = \{ (x_1, \mcMhat(x_0, y_0, f_0)) \mid \dots\} \to W(1)  \\
 (x_1, \mcMhat(x_0, y_0, f_0))\quad  \mapsto \quad \left(\Ind(x_1) ,\left[ \begin{aligned} \Ind(x_0) \\ \Ind(y_0) \end{aligned} \right] \right)
\end{gather*}
and generally for $1 \leq l \leq n$ 
\begin{gather*}
\mcG: X(l) \to W(l) \\
\left(a_l, \mcMhat\bigl(a_{l-1}, b_{l-1}, f_{l-1
\left[
\begin{smallmatrix}
{a_{l-2}, \dots, a_0} \\
{b_{l-2}, \dots, b_0}
\end{smallmatrix}
\right]
}\bigr)\right)
\mapsto
\left(
\Ind(a_l), \left[ \begin{smallmatrix}
            \Ind(a_{l-1}), \dots, \Ind(a_0) \\
            \Ind(b_{l-1}), \dots, \Ind(b_0)
           \end{smallmatrix}
\right]\right).
\end{gather*}

\begin{Theorem}
\label{wfunctor}
$\mcG$ is an almost strict $n$-functor from $\mcX$ to $\mcW$.
\end{Theorem}

\begin{proof}[Proof of \refhomfunctor\ and \refwfunctor]
Using \refVW\ we can proof \refhomfunctor\ and \refwfunctor\ simultanously, i.e. it is sufficient to show the claim for $\mcF$.

{\em (a) $\mcF$ is compatible with source and target functions:} Let $A_l \in X(l)$ with
$$ 
A_l =
\left(a_l, \mcMhat\bigl(a_{l-1}, b_{l-1}, f_{l-1
\left[
\begin{smallmatrix}
{a_{l-2}, \dots, a_0} \\
{b_{l-2}, \dots, b_0}
\end{smallmatrix}
\right]
}\bigr)\right)
$$ 
and compute for the source function
\begin{align*}
 \mcF( s( A_l))& = \mcF 
 \left(a_{l-1}, \mcMhat\bigl(a_{l-2}, b_{l-2}, f_{l-2
\left[
\begin{smallmatrix}
{a_{l-3}, \dots, a_0} \\
{b_{l-3}, \dots, b_0}
\end{smallmatrix}
\right]
}\bigr)\right) \\
& = \left(
\Ind(a_{l-1}), \left[ \begin{smallmatrix}
            \Ind(a_{l-2}), \dots, \Ind(a_0) \\
            \Ind(b_{l-2}), \dots, \Ind(b_0)
           \end{smallmatrix}
\right]\right) \\
& = s\left(
\Ind(a_{l}), \left[ \begin{smallmatrix}
            \Ind(a_{l-1}), \dots, \Ind(a_0) \\
            \Ind(b_{l-1}), \dots, \Ind(b_0)
           \end{smallmatrix}
\right]\right) \\
& = s(\mcF(A_l)).
\end{align*} 
The claim for the target function follows similarly.


\vsp

{\em (b) $\mcF$ is compatible with the identity functions:} Let $A_l \in X(l)$ as in {\em (a)} and compute 
\begin{align*}
 \mcF(\mathbf 1 (A_l)) & = \mcF 
 \left(a_l, \mcMhat\bigl(a_{l}, a_l, f_{l
\left[
\begin{smallmatrix}
{a_{l-1}, \dots, a_0} \\
{b_{l-1}, \dots, b_0}
\end{smallmatrix}
\right]
}\bigr)\right)
=
\left(
0, \left[ \begin{smallmatrix}
            \Ind(a_{l}), \Ind(a_{l-1}), \dots, \Ind(a_0) \\
            \Ind(a_{l}), \Ind(b_{l-1}), \dots, \Ind(b_0)
           \end{smallmatrix}
\right]\right) \\
& = \mathbf 1  
\left(
\Ind(a_{l}), \left[ \begin{smallmatrix}
            \Ind(a_{l-1}), \dots, \Ind(a_0) \\
            \Ind(b_{l-1}), \dots, \Ind(b_0)
           \end{smallmatrix}
\right]\right)    
= \mathbf 1 (\mcF (A_l)).
\end{align*}


\vsp

{\em (c) $\mcF$ is compatible with the composite:} Let $(C_l, A_l) \in X(l)\x_p X(l)$ with
\begin{align*}
A_l & = \left(a_l, \mcMhat(a_{l-1}, b_{l-1}, f_{l-1
\left[
\begin{smallmatrix}
a_{l-2}, \dots, a_{p+1}, x_p, \al_{p-1}, \dots, \al_0 \\
b_{l-2}, \dots, b_{p+1}, y_p, \be_{p-1}, \dots, \be_0
\end{smallmatrix}
\right]
})
\right), \\
C_l & = \left(c_l, \mcMhat(c_{l-1}, d_{l-1}, f_{l-1
\left[
\begin{smallmatrix}
c_{l-2}, \dots, c_{p+1}, y_p, \al_{p-1}, \dots, \al_0 \\
d_{l-2}, \dots, d_{p+1}, z_p, \be_{p-1}, \dots, \be_0
\end{smallmatrix}
\right]
})
\right)
\end{align*}
and compute
\begin{align*}
& \mcF(C_l \circ_p A_l) \\
& = \mcF \left((a_l, c_l), \mcMhat\bigl((a_{l-1}, c_{l-1}), (b_{l-1}, d_{l-1}), f_{l-1
\left[
\begin{smallmatrix}
(a_{l-2}, c_{l-2}), \dots,  (a_{p+1}, c_{p+1}), x_p, \al_{p-1}, \dots, \al_0 \\
(b_{l-2}, d_{l-2}), \dots,(b_{p+1}, d_{p+1}), z_p, \be_{p-1}, \dots, \be_0  
\end{smallmatrix}
\right]
}\bigr)
\right) \\
& = 
\left(
\begin{smallmatrix} \Ind(a_{l}) + \Ind(c_l) \end{smallmatrix}, \left[ \begin{smallmatrix}
            (\Ind(a_{l-1}) + \Ind(c_{l-1})), \dots, (\Ind(a_{p+1}) + \Ind(c_{p+1})), \Ind(x_p), \Ind(\al_{p-1}), \dots, \Ind(\al_0) \\
            (\Ind(b_{l-1}) + \Ind(d_{l-1})), \dots, (\Ind(b_{p+1}) + \Ind(d_{p+1})), \Ind(z_p), \Ind(\be_{p-1}), \dots, \Ind(\be_0)
           \end{smallmatrix}
\right]\right) \\
& = 
\left(
\Ind(c_{l}), \left[ \begin{smallmatrix}
            \Ind(c_{l-1}), \dots, \Ind(c_{p+1}), \Ind(y_p), \Ind(\al_{p-1}), \dots, \Ind(\al_0) \\
            \Ind(d_{l-1}), \dots, \Ind(d_{p+1}), \Ind(z_p), \Ind(\be_{p-1}), \dots, \Ind(\be_0)
           \end{smallmatrix}
\right]\right)  \\
& \qquad \circ_p
\left(
\Ind(a_{l}), \left[ \begin{smallmatrix}
            \Ind(a_{l-1}), \dots, \Ind(a_{p+1}), \Ind(x_p), \Ind(\al_{p-1}), \dots, \Ind(\al_0) \\
            \Ind(b_{l-1}), \dots, \Ind(b_{p+1}), \Ind(y_p), \Ind(\be_{p-1}), \dots, \Ind(\be_0)
           \end{smallmatrix}
\right]\right) \\ 
& = \mcF(C_l) \circ_p \mcF(A_l).
\end{align*}


\vsp

{\em (d) $\mcF$ is compatible with the index requirements of the elements in the globular set $V$:} In fact, the index requirements of $V$ were set up with $X$ in mind. Consider
\begin{align*}
\left(a_l, \mcMhat\bigl(a_{l-1}, b_{l-1}, f_{l-1
\left[
\begin{smallmatrix}
{a_{l-2}, \dots, a_0} \\
{b_{l-2}, \dots, b_0}
\end{smallmatrix}
\right]
}\bigr)\right)
=A_l 
\in X(l)
\end{align*}
with $a_{l-1} \neq b_{l-1}$ and note 
$$\dim \mcM\bigl(a_{l-1}, b_{l-1}, f_{l-1
\left[
\begin{smallmatrix}
{a_{l-2}, \dots, a_0} \\
{b_{l-2}, \dots, b_0}
\end{smallmatrix}
\right]
}\bigr)
= \Ind(a_{l-1}) - \Ind(b_{l-1})
$$ 
and 
$$\dim \mcMhat \bigl(a_{l-1}, b_{l-1}, f_{l-1
\left[
\begin{smallmatrix}
{a_{l-2}, \dots, a_0} \\
{b_{l-2}, \dots, b_0}
\end{smallmatrix}
\right]
}\bigr)
= \Ind(a_{l-1}) - \Ind(b_{l-1}) -1.
$$
If $a_{l-1} \neq b_{l-1}$ then $\Ind(a_{l-1}) - \Ind(b_{l-1}) \geq 1$ since we are considering a negative gradient flow. If $a_{l-1} = b_{l-1}$ then we formally consider both $\mcM(a_{l-1}, a_{l-1}, \dots)$ and $\mcMhat(a_{l-1}, a_{l-1}, \dots)$ as zero dimensional.
Since $a_l$ lives on the unparametrized space $\mcMhat(a_{l-1}, b_{l-1}, \dots)$ we have $0 \leq \Ind(a_l) < \Ind(a_{l-1}) -\Ind(b_{l-1})$ for all $1 \leq l \leq n$. This satisfies the index conditions in the definition of $V(l)$ and concludes the proof of \refhomfunctor\ and \refwfunctor.
\end{proof}


\section{Example: The $2$-torus $\mathbb T^2$}

Consider the $2$-torus $\mathbb T^2 = \R^2\slash \Z^2$ with the flat metric and the Morse function $f_0(x,y)=\cos(2\pi x) + \cos (2 \pi y)$ whose critical points are $\{(\frac{k}{2}, \frac{l}{2}) \mid k,l \in \Z\}$. Let us work on the fundamental domain $[0,1] \x [0,1]$ which has four critical points $w:=w_0=(0,0)=(1,0)=(0,1)=(1,1)$ and $x:=x_0=(\frac{1}{2}, 0)=(\frac{1}{2}, 1)$ and $y:=y_0=(0, \frac{1}{2})=(1, \frac{1}{2})$ and $z:=z_0=(\frac{1}{2}, \frac{1}{2})$ as in Figure \ref{2torus}. We suppress the level indices in $w_0, x_0, y_0, z_0$ since it would complicate the notation.

\begin{figure}[h] 
\label{2torus}

\begin{center}

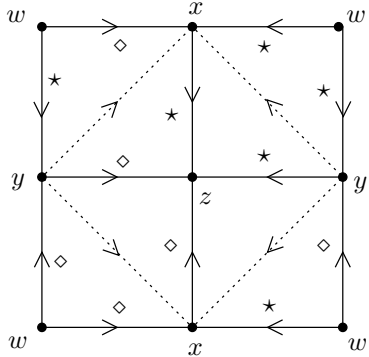
\caption{Morse trajectories on $\mathbb T^2$}

\end{center}

\end{figure}


We have $\Ind(w)=2$, $\Ind(x)=\Ind(y)=1$ and $\Ind(z)=0$ and the moduli spaces $\mcMhat(w,x)$, $\mcMhat(w,y)$, $\mcMhat(x,z)$ and $\mcMhat(y,z)$ are zero dimensional and have two connected components each. We denote them by $\mcMhat(w,x)=\mcMhat(w, x)_\diamond \cup \mcMhat(w,x)_\star$ etc. 

\textit{If we consider a (component of a) zero dimensional moduli space as a \textbf{point} instead of a \textbf{space}, we write $\mfmhat(\cdot, \cdot)$ instead of $\mcMhat(\cdot, \cdot)$} in the following.

\subsection{The almost strict Morse $n$-category}

$\mcMhat(w,z)$ is $1$-dimensional and has four connected components. We choose a Morse function 
$f_{1 \left[  \begin{smallmatrix} w \\z \end{smallmatrix} \right]}$
on $\mcMhat(w,z)$ which is strictly monotone and has its critical points on the endpoints of the intervals. Let $f_{1 \left[  \begin{smallmatrix} w \\z \end{smallmatrix} \right]}$ be maximal on {\em the critical point} $(\mfmhat(w,y)_i, \mfmhat(y,z)_j)$ and minimal on {\em the critical point} $(\mfmhat(w,x)_i, \mfmhat(x,z)_j)$ for $i$, $j \in \{\diamond ,\star \}$.
We have
\beqs
X(0)=\{w, x, y, z \}
\eeqs
and we compute
\beqs
X(1)=
\left\{
\begin{gathered}
\left.
\begin{gathered}
(\mfmhat(w,x)_i, \mcMhat(w,x)), (\mfmhat(w,y)_i, \mcMhat(w,y)), \\
(\mfmhat(x,z)_i, \mcMhat(x,z)), (\mfmhat(y,z)_i, \mcMhat(y,z)), \\
((\mfmhat(w,x)_i, \mfmhat(x,z)_j), \mcMhat(w,z)), \\
((\mfmhat(w,y)_i, \mfmhat(y,z)_j), \mcMhat(w,z))
\end{gathered}
\right| i, j \in \{\diamond, \star \}
\end{gathered}
\right\}.
\eeqs
and obtain
\begin{align*}
X(1)\x_0 X(1)& =\{(\xiti, \xi) \in X(1) \x X(1) \mid s(\xiti) = t(\xi)\} \\
& 
=
\left\{
\begin{gathered}
\left.
\begin{gathered}
\left( (\mfmhat(x,z)_i, \mcMhat(x,z)), (\mfmhat(w,x)_j, \mcMhat(w,x)) \right), \\
\left( (\mfmhat(y,z)_i, \mcMhat(y,z)), (\mfmhat(w,y)_j, \mcMhat(w,y)) \right)
\end{gathered}
\right|
i, j \in \{\diamond, \star\}
\end{gathered}
\right\}
\end{align*}
and concatenate exemplarily
\beqs
\left(\mfmhat(w, x)_i, \mcMhat(w,x)\right) \circ_0 \left(\mfmhat(x,z)_j, \mcMhat(x,z)\right) 
= \left( (\mfmhat(w,x)_i, \mfmhat(x,z)_j), \mcMhat(w,z)\right).
\eeqs
We set
\begin{align*}
 \blacksquare &:= \bigl((\mfmhat(w,y)_\star, \mfmhat(y,z)_\diamond),(\mfmhat(w,x)_\diamond, \mfmhat(x,z)_\star)\bigr), \\
 \blacktriangle & := \bigl((\mfmhat(w,y)_\diamond, \mfmhat(y,z)_\diamond), (\mfmhat(w,x)_\diamond, \mfmhat(x,z)_\diamond)\bigr), \\
 \spadesuit & := \bigl((\mfmhat(w,y)_\diamond, \mfmhat(y,z)_\star), (\mfmhat(w,x)_\star, \mfmhat(x,z)_\diamond)\bigr), \\
 \clubsuit & := \bigl((\mfmhat(w,y)_\star, \mfmhat(y,z)_\star), (\mfmhat(w,x)_\star, \mfmhat(x,z)_\star)\bigr)
\end{align*}
and compute
\beqs
X(2)=
\left\{
\begin{gathered}
\left.
\begin{gathered}
\blacktriangledown_{wx}:= \left(\mfmhat(w,x)_i, \mcMhat(\mfmhat(w,x)_i, \mfmhat(w,x)_i   ) \right), \\
\blacktriangledown_{wy}:=\left(\mfmhat(w,y)_i, \mcMhat(\mfmhat(w,y)_i, \mfmhat(w,y)_i)  \right), \\
\blacktriangledown_{xz}:= \left( \mfmhat(x,z)_i, \mcMhat(\mfmhat(x,z)_i, \mfmhat(x,z)_i )  \right),\\
\blacktriangledown_{yz}:= \left( (\mfmhat(y,z)_i, \mcMhat(\mfmhat(y,z)_i, \mfmhat(y,z)_i ) \right), \\
\left(\mfmhat(\blacksquare), \mcMhat (  \blacksquare ) \right), 
\left(\mfmhat(\blacktriangle), \mcMhat( \blacktriangle )\right), \\
\left(\mfmhat(\spadesuit),\mcMhat( \spadesuit) \right), 
\left(\mfmhat(\clubsuit), \mcMhat(  \clubsuit)\right)
\end{gathered}
\right|
i \in \{\diamond, \star \}
\end{gathered}
\right\}.
\eeqs
We calculate for $i \in \{\diamond, \star\}$ and $q \in \{x,y\}$:
\begin{align*}
& s^2\left(\mfmhat(w,q)_i, \mcMhat(\mfmhat(w,q)_i, \mfmhat(w,q)_i   ) \right) 
= s \left( \mfmhat(w,q)_i, \mcMhat(w,q)  \right) = w,	 \\
& t^2\left(\mfmhat(w,q)_i, \mcMhat(\mfmhat(w,q)_i, \mfmhat(w,q)_i   ) \right) 
= t \left( \mfmhat(w,q)_i , \mcMhat(w,q) \right) = q, \\
& s^2 \left( \mfmhat(q,z)_i, \mcMhat(\mfmhat(q,z)_i, \mfmhat(q,z)_i )  \right)
= s \left( \mfmhat(q,z)_i, \mcMhat(q,z) \right)= q ,\\
& t^2 \left( \mfmhat(q,z)_i, \mcMhat(\mfmhat(q,z)_i, \mfmhat(q,z)_i )  \right)
= t \left(\mfmhat(q,z)_i, \mcMhat( q,z) \right) = z , \\
\end{align*}
and
\begin{align*}
& s^2 \left( (\mfmhat((\mfmhat(w,y)_\star, \mfmhat(y,z)_\diamond),(\mfmhat(w,x)_\diamond, \mfmhat(x,z)_\star)), \mcMhat ( \blacksquare) \right) \\
& \quad = s \left( (\mfmhat(w,y)_\star, \mfmhat(y,z)_\diamond), \mcMhat(w,z ) \right) \\
& \quad = w ,\\
& s^2 \left(\mfmhat((\mfmhat(w,y)_\diamond, \mfmhat(y,z)_\diamond), (\mfmhat(w,x)_\diamond, \mfmhat(x,z)_\diamond)), \mcMhat( \blacktriangle )\right) \\
& \quad = s \left( (\mfmhat(w,y)_\diamond, \mfmhat(y,z)_\diamond), \mcMhat(w,z ) \right) \\
& \quad = w, \\
& s^2 \left(\mfmhat((\mfmhat(w,y)_\diamond, \mfmhat(y,z)_\star), (\mfmhat(w,x)_\star, \mfmhat(x,z)_\diamond)), \mcMhat( \spadesuit) \right) \\
& \quad = s \left( (\mfmhat(w,y)_\diamond, \mfmhat(y,z)_\star), \mcMhat(w,z)  \right) \\
& \quad = w , \\
& s^2 \left(\mfmhat((\mfmhat(w,y)_\star, \mfmhat(y,z)_\star), (\mfmhat(w,x)_\star, \mfmhat(x,z)_\star)), \mcMhat( \clubsuit) \right) \\
& \quad = s \left( (\mfmhat(w,y)_\star, \mfmhat(y,z)_\star), \mcMhat(w,z) \right) \\
& \quad = w
\end{align*}
and
\begin{align*}
& t^2 \left( (\mfmhat((\mfmhat(w,y)_\star, \mfmhat(y,z)_\diamond),(\mfmhat(w,x)_\diamond, \mfmhat(x,z)_\star)), \mcMhat ( \blacksquare) \right) \\
& \quad = t \left( (\mfmhat(w,x)_\diamond, \mfmhat(x,z)_\star), \mcMhat(w,z ) \right) \\
& \quad = z ,\\
& t^2 \left(\mfmhat((\mfmhat(w,y)_\diamond, \mfmhat(y,z)_\diamond), (\mfmhat(w,x)_\diamond, \mfmhat(x,z)_\diamond)), \mcMhat( \blacktriangle )\right) \\
& \quad = t \left(  (\mfmhat(w,x)_\diamond, \mfmhat(x,z)_\diamond) , \mcMhat(w,z ) \right) \\
& \quad = z, \\
& t^2 \left(\mfmhat((\mfmhat(w,y)_\diamond, \mfmhat(y,z)_\star), (\mfmhat(w,x)_\star, \mfmhat(x,z)_\diamond)), \mcMhat( \spadesuit) \right) \\
& \quad = t \left(  (\mfmhat(w,x)_\star, \mfmhat(x,z)_\diamond) , \mcMhat(w,z)  \right) \\
& \quad = z , \\
& t^2 \left(\mfmhat((\mfmhat(w,y)_\star, \mfmhat(y,z)_\star), (\mfmhat(w,x)_\star, \mfmhat(x,z)_\star)), \mcMhat( \clubsuit) \right) \\
& \quad = t \left( (\mfmhat(w,x)_\star, \mfmhat(x,z)_\star)  , \mcMhat(w,z) \right) \\
& \quad = z
\end{align*}
This yields 
\begin{align*}
&X(2) \x_1 X(2)  \\
& =\{(\xiti, \xi) \in X(2) \x X(2) \mid t(\xi)=s(\xiti)\} \\
&= \left\{
\begin{gathered}
\left( \left(\mfmhat(w,q)_i, \mcMhat(\mfmhat(w,q)_i, \mfmhat(w,q)_i   ) \right), \left(\mfmhat(w,q)_i, \mcMhat(\mfmhat(w,q)_i, \mfmhat(w,q)_i   ) \right) \right) \\
\left( \left( \mfmhat(q,z)_i, \mcMhat(\mfmhat(q,z)_i, \mfmhat(q,z)_i )  \right), \left( \mfmhat(q,z)_i, \mcMhat(\mfmhat(q,z)_i, \mfmhat(q,z)_i )  \right)  \right) \\
\mbox{for } i \in \{\diamond, \star \}, q \in \{x,y\}
\end{gathered}
\right\}
\end{align*}
where we compute for example
\begin{align*}
& \left(\mfmhat(w,q)_i, \mcMhat(\mfmhat(w,q)_i, \mfmhat(w,q)_i   ) \right) \circ_1 \left(\mfmhat(w,q)_i, \mcMhat(\mfmhat(w,q)_i, \mfmhat(w,q)_i   ) \right) \\
& \quad = 
\left(
(\mfmhat(w,q)_i,\mfmhat(w,q)_i), \mcMhat( (\mfmhat(w,q)_i,\mfmhat(w,q)_i), (\mfmhat(w,q)_i,\mfmhat(w,q)_i) )
\right) \\
& \quad \simeq
\left(
\mfmhat(w,q)_i, \mcMhat(\mfmhat(w,q)_i,\mfmhat(w,q)_i) 
\right)
\end{align*}
since the underlying space is a singleton. We find
\begin{align*}
& X(2)\x_0 X(2) \\
& \quad  =\{(\xiti, \xi) \in X(2) \x X(2) \mid t^2(\xi)=s^2(\xiti)\} \\
& \quad = \left\{
\begin{gathered}
\left( 
\left( \mfmhat(q,z)_i, \mcMhat(\mfmhat(q,z)_i, \mfmhat(q,z)_i )  \right)
, \left(\mfmhat(w,q)_i, \mcMhat(\mfmhat(w,q)_i, \mfmhat(w,q)_i   ) \right) 
\right) \\
\mbox{for } i \in \{\diamond , \star \}, q \in \{x,y\}
\end{gathered}
\right\}
\end{align*}
and we compute 
\begin{align*}
& \left( \mfmhat(q,z)_i, \mcMhat(\mfmhat(q,z)_i, \mfmhat(q,z)_i )  \right) \circ_0 
\left(\mfmhat(w,q)_i, \mcMhat(\mfmhat(w,q)_i, \mfmhat(w,q)_i   ) \right) \\
& \quad = 
\left( ( \mfmhat(w,q)_i ,\mfmhat(q,z)_i), \mcMhat(  (\mfmhat(w,q)_i, \mfmhat(q,z)_i), (\mfmhat(w,q)_i, \mfmhat(q,z)_i ))
\right).
\end{align*}
Note that all elements of $X(l)$ for $l \geq 3$ will be of the form $(\xi, \mcMhat(\xi, \xi))$ such that they do not contribute any new information.


\subsection{The image of the Morse $n$-category under the functors}

The elements of $X(0)$, $X(1)$ and $X(2)$ look quite complicated. Let us now consider their image under $\mcF$ and $\mcG$. For sake of readability, we will allways use the short notation introduced in \refVW\ which expresses the functor $\mcF$ in terms of $\mcG$. Thus it is sufficient to calculate the image of $\mcX$ under $\mcG$. 

\vsp

For $X(0)=\{w, x, y, z \}$, we obtain
\beqs
\mcG(w)=2, \quad \mcG(x)=1, \quad  \mcG(y)=1, \quad  \mcG(z)=0.
\eeqs
which are the indices of the critical points. Now consider $X(1)$ and obtain for the first four elements
\begin{align*}
 \mcG(\mfmhat(w,x)_i, \mcMhat(w,x)) & = \left(0 \left[\begin{matrix} 2 \\1 \end{matrix} \right] \right), \\
 \mcG (\mfmhat(w,y)_i, \mcMhat(w,y))& = \left(0 \left[\begin{matrix} 2 \\1 \end{matrix} \right] \right), \\
 \mcG(\mfmhat(x,z)_i, \mcMhat(x,z)) & = \left(0 \left[\begin{matrix} 1 \\0 \end{matrix} \right] \right), \\
\mcG (\mfmhat(y,z)_i, \mcMhat(y,z)) & = \left(0 \left[\begin{matrix} 1 \\0 \end{matrix} \right] \right) 
\end{align*}
and for the last two elements
\begin{align*}
\mcG((\mfmhat(w,x)_i, \mfmhat(x,z)_j), \mcMhat(w,z))& = \left(0 \left[\begin{matrix} 2 \\0 \end{matrix} \right] \right), \\
 \mcG((\mfmhat(w,y)_i, \mfmhat(y,z)_j), \mcMhat(w,z))& = \left(1 \left[\begin{matrix} 2 \\0 \end{matrix} \right] \right).
\end{align*}
For the elements of $X(2)$, we calculate
\begin{align*}
 \mcG(\blacktriangledown_{wx})=\mcG(\blacktriangledown_{wy}) & = \left(0 \left[\begin{matrix} 0 & 2 \\0 & 1 \end{matrix} \right] \right), \\
 \mcG(\blacktriangledown_{xz}) =\mcG(\blacktriangledown_{yz}) & =\left(0 \left[\begin{matrix} 0 & 1 \\0 & 0 \end{matrix} \right] \right), \\
 \mcG(\blacksquare, \mcMhat(\blacksquare)) = \mcG(\blacktriangle, \mcMhat(\blacktriangle)) = \mcG(\spadesuit, \mcMhat(\spadesuit)) = \mcG(\clubsuit, \mcMhat(\clubsuit)) & = \left(0 \left[\begin{matrix} 1 & 2 \\0 & 0 \end{matrix} \right] \right).
\end{align*}
Thus the functors simplify the picture considerably by providing an overview of the history of the indices of the critical points resp.\ the dimension of the involved moduli spaces.


\end{document}